\documentclass{article}    

\usepackage[a4paper, total={6in, 9in}]{geometry}
\usepackage{amssymb,dsfont,color}

\let\theoremstyle\relax
\let\theoremstyle\relax

\usepackage{graphicx}
\usepackage{amsthm}
\usepackage{epsfig}
\usepackage{pstricks}
\usepackage{subcaption}
\usepackage{soul}
\usepackage{amsmath}
\usepackage{mathrsfs}
\usepackage{enumerate}
\usepackage{mathtools}
\usepackage{comment}

\newtheorem{theorem}{Theorem}
\newtheorem{proposition}[theorem]{Proposition}

\newtheorem{remark}[theorem]{Remark}
\newtheorem{lemma}[theorem]{Lemma}

\newtheorem{definition}[theorem]{Definition}

\begin{document}

\title{Ensemble control of n-level quantum systems with a scalar control} 
\date{}
\author{Ruikang Liang, Ugo Boscain,  and~Mario~Sigalotti
\thanks{The three authors are with
Laboratoire Jacques-Louis Lions, Sorbonne Universit\'e, Universit\'e de
Paris, CNRS, Inria, Paris, FranceSorbonne Universit\'e, Universit\'e de
Paris Cit\'e, CNRS, INIRIA, Laboratoire Jacques-Louis Lions, LJLL, F-75005 Paris, France
  {\tt 
   ruikang.liang@sorbonne-universite.fr, 
  ugo.boscain@sorbonne-universite.fr,
mario.sigalotti@inria.fr}.
 }
}
\maketitle

\begin{abstract}                          
 In this paper we discuss how a general bilinear finite-dimensional closed quantum system with dispersed parameters can be steered between eigenstates. 
      We show that, under suitable conditions on the separation of spectral gaps and the boundedness of parameter dispersion, rotating wave and adiabatic approximations can be employed in cascade to achieve population inversion between arbitrary eigenstates. We propose an explicit control law and test numerically the sharpness of the conditions on several examples. 
\end{abstract}
 \section{Introduction}
Let us consider a continuum of $n$-level closed quantum systems described by the Schr\"odinger equation
\begin{equation}
\label{system:n_level}
    i\dot{\psi}(t)=(H(\alpha)+\omega(t)H_{c}{ (\delta)})\psi(t),
\end{equation}
where $\omega(\cdot)$ is a real-valued control. Here the Hamiltonian $H(\alpha)$ is determined by an unknown parameter $\alpha$ taking values in a closed and connected subdomain $\mathcal{D}$ of $\mathbb{R}^{m}$.
We assume that $H(\alpha)$ 
has the structure 
\begin{equation*}
    H(\alpha)=
    \begin{pmatrix}
        \lambda_{1}(\alpha) &0 &\dots  & 0\\
        0& \lambda_{2}(\alpha) & 0 &\vdots\\
        \vdots&0&\ddots&0\\
        0&\dots&0&\lambda_{n}(\alpha)
   \end{pmatrix},
\end{equation*}
where $\lambda_{j}:\mathcal{D}\rightarrow\mathbb{R}$ is a continuous function for each $j=1,\dots,n$.
The matrix $H_{c}$ is 
self-adjoint 
and 
describes the control coupling
between the eigenstates of the system. It has the form 
 \begin{equation*} H_{c}=
{\begin{pmatrix}
         \delta_{11} & \delta_{12} & \dots &\delta_{1n}\\
         \delta_{12} & \delta_{22} & \ddots &  \vdots \\
         \vdots &\ddots & \ddots& \delta_{n-1,n}
         \\         \delta_{1n} &\cdots &\delta_{n-1,n} & \delta_{n,n}
     \end{pmatrix}}
 \end{equation*}
where the $\delta_{jk}$ are real parameters.
We assume that each $\delta_{jk}$ is unknown but that it belongs to some closed interval 
$\mathcal{I}_{jk}=\left[\delta_{jk}^{0},\delta_{jk}^{1}\right]$ 
in $\mathbb{R}$.
We are going to consider the 
\textit{ensemble control problem} consisting in steering the continuum of dynamics \eqref{system:n_level} with the same control signal, independent of $\alpha$ and of the $\delta_{jk}$.

{Our main tool to tackle the problem is the adiabatic theory which states that, for a slowly varying Hamiltonian, the spectral subspaces 
are approximately invariant under the associated dynamics (see \cite{Teufel}). 

The problem of ensemble controllability for quantum systems has been mainly developed for two-level systems steered by two 
real-valued
controls. In experimental situations it is often takled  via direct optimization methods (see, for instance, \cite{long000,glaser2015training} and also the review papers \cite{rev1,rev2}). The mathematical aspects of the ensemble controllability problem have been addressed, in particular, in \cite{BeauchardCoronRouchon,li2006ensemble,li2009ensemble}. In these papers the authors show that a uniform control can be constructed to steer the system from a common initial state to an arbitrary set of target states continuously parameterized by the uncertainties of the system, i.e., $(\alpha,\delta)$. Let us also mention  \cite{AugierSIAM,AugierMCRF,LeghtasSarletteRouchon}
for some further results on ensemble control of quantum systems by adiabatic motion
and \cite{ChittaroGauthier,Netoetal} for  related ensemble stabilization problems for two-level quantum systems. 
For systems with a single scalar input, as described in system \eqref{system:n_level}, the strategies used in the papers above cannot be applied directly due to the lack of degrees of freedom. In the physics literature, the most common way to approximate an effective Hamiltonian with several degrees of freedom using a single input is through the rotating wave approximations.
However, this approach is not always compatible with the ensemble adiabatic motion, as observed in \cite{ROBIN2022414}.
The article \cite{ROBIN2022414} proposes a rigorous constructive approach that combines, in a suitable range of dispersion parameters, rotating wave and adiabatic approximations, extending the results for 2-level systems \cite{BeauchardCoronRouchon,li2006ensemble,li2009ensemble}  to the case of scalar controls. Results for the ensemble control of $n$-level systems with a scalar input have been obtained in \cite{augier2022effective}. The latter article uses a different combination of adiabatic and rotating wave approximations, which is robust with respect to coupling strengths in $H_{c}(\delta)$ but not to dispersions in the frequencies of $H(\alpha)$.
Up to now the problem of ensemble control of a $n$-level systems with a scalar control with dispersions in the frequencies of $H(\alpha)$ remained open. This is what we are studying in the present paper by extending the approach developed 
in \cite{ROBIN2022414}.
 The analysis is here much more complicated since the presence of several energy levels 
increases the possible resonances.}

We employ the rotating wave approximation and an adiabatic approximation in cascade to realise population inversion between an arbitrary pair of eigenstates of $H(\alpha)$. High-order averaging results are necessary in both the rotating wave and adiabatic steps. Otherwise 
the fidelity of their cascade can fail to converge to 1 due to the competing time scales of the two approximations. We highlight the importance of having no overlaps between the uncertainty intervals of each frequency and that, assuming further higher-order conditions on the resonance frequencies, it is possible to steer the system with higher precision.

The article is organized as follows: In Section~\ref{sec:results}, we state the main result (Theorem~ \ref{theorem:main}) and we provide an explicit construction of the control law. 
We provide a sketch of the main steps of the proof of Theorem~ \ref{theorem:main} in Section~\ref{sec:sketch-proof}.
The proof itself is given in  Section~\ref{sec:proof}, 
where fast-oscillating terms are first eliminated through successive time-dependent changes of variables and rotating wave approximations. A non-standard adiabatic theorem is then applied to the system. In Section~\ref{sec:numerical}, we test numerically the sharpness of the proposed conditions in the case of a 4-level system.

\section{Results}\label{sec:results}
To realize a population inversion between two arbitrary eigenstates, we will consider a time scale $\epsilon_{1}$ for the rotating wave approximation and another time scale $\epsilon_{2}$ for the adiabatic following. The control law in our algorithm will be a chirped pulse of the type
\begin{equation}
\label{eq:control}
    \omega_{\epsilon_{1},\epsilon_{2}}(t)=2\epsilon_{1}u(\epsilon_{1}\epsilon_{2}t)\cos\left(\int_{0}^{t}f(\epsilon_1\epsilon_2\tau)\text{d}\tau\right),
\end{equation}
where $u,f:[0,T]\to \mathbb{R}$ are functions to be chosen. The goal is to induce 
a transition from the initial state $\textbf{e}_{p}$ towards a state of the form $\exp(i\theta)\textbf{e}_{q}$ in time $T/(\epsilon_{1}\epsilon_{2})$
when $\epsilon_{1}$ and $\epsilon_{2}$ are small.

In the following, we will denote by $\left\{\textbf{e}_{j}\right\}_{j\in\{1,\dots,n\}}$ the canonical basis of $\mathbb{R}^{n}$ and $\left\{\textbf{e}_{jk}\right\}_{(j,k)\in\{1,\dots,n\}^2}$ the canonical basis of $M_{n}(\mathbb{R})$ (space of real $n\times n$ matrices). 
\begin{theorem}\label{theorem:main}
    Let us assume that 
    for all $1\leq j<k\leq n$, and for all $\alpha\in \mathcal{D}$,  
    $\lambda_{k}(\alpha)-\lambda_{j}(\alpha)>0$. 
Fix  $1\leq p<q\leq n$. Assume that $\delta_{pq}$ belongs to a closed interval $\mathcal{I}_{pq}=\left[\delta_{pq}^{0},\delta_{pq}^{1}\right]$ such that $0\notin\mathcal{I}_{pq}$ and there exist  $0<v_{0}<v_{1}$ such that 
\begin{enumerate}
    \item \label{cond_1}For all $\alpha\in\mathcal{D}$, $\lambda_{q}(\alpha)-\lambda_{p}(\alpha)\in(v_0,v_1)$;
    \item \label{cond_2}For all $1\leq j<k\leq n$ such that $(j,k)\neq(p,q)$ and all $\alpha\in\mathcal{D}$, we have $\lambda_{k}(\alpha)-\lambda_{j}(\alpha)\notin\left[v_0,v_1\right]$.
\end{enumerate}
    Fix $T>0$ and take $u,f\in\mathcal{C}^{2}([0,T],\mathbb{R})$ such that
    \begin{enumerate}
        \item[i)] $(u(0),f(0))=(0,v_{0})$ and $(u(T),f(T))=(0,v_{1})$;
        \item[ii)] $\forall s\in (0,T),u(s)>0$ and $\forall s\in[0,T],\dot{f}(s)>0$.
    \end{enumerate}
    Denote by $\psi_{\epsilon_{1},\epsilon_{2}}$ the solution of \eqref{system:n_level} with  initial condition $\psi_{\epsilon_{1},\epsilon_{2}}(0)=\textbf{e}_{p}$ and the control law $\omega_{\epsilon_{1},\epsilon_{2}}$ as in \eqref{eq:control}. Then there exist $C>0$ and $\eta>0$ such that for every $\alpha\in\mathcal{D}$ and every $(\epsilon_{1},\epsilon_{2})\in(0,\eta)^{2}$,
\begin{equation*}
\begin{aligned}
\Big\|\psi_{\epsilon_{1},\epsilon_{2}}\left(\frac{T}{\epsilon_1\epsilon_2}\right)-\exp(i\theta)\textbf{e}_{q}\Big\|\leq C(\epsilon_2\epsilon_1^{-1}+\epsilon_{1}^{3/2}\epsilon_2^{-1/2}+\epsilon_1+\epsilon_1^{5/2}\epsilon_2^{-3/2}),
\end{aligned}
\end{equation*}
for some $\theta\in\mathbb{R}$.
\end{theorem}
\begin{remark}
    If we choose $\epsilon_{2}=\epsilon_{1}^{7/5}$, then there exists $C'>0$ such that for every $\alpha\in\mathcal{D}$ and $(\epsilon_1,\epsilon_2)\in(0,\eta)^2$,
\begin{equation*}\min_{\theta\in[0,2\pi]}\left\|\psi_{\epsilon_{1},\epsilon_{2}}\left(\frac{T}{\epsilon_{1}\epsilon_{2}}\right)-\exp(i\theta)\textbf{e}_{q}\right\|\leq C'\epsilon_{1}^{2/5},
    \end{equation*}
    that is, the final state is arbitrarily close to the eigenstate $\textbf{e}_q$, up to a phase, as $\epsilon_1$ goes to zero. 
\end{remark}
\begin{remark}
\label{remark:control_law_construction}
   Here we fix $T=1$ and give a simple construction of the control law satisfying the conditions of Theorem \ref{theorem:main} by choosing $u(s)=\sin(\pi s)$ and $f(s)=v_0+s(v_1-v_0)$ for $s\in[0,1]$.
Thus the control law is given by
\begin{equation*}
\label{eq:control_law}\omega_{\epsilon_{1},\epsilon_{2}}(t)=2\epsilon_{1}\sin(\epsilon_{1}\epsilon_{2}\pi t)\cos\Big(v_0t+\frac{\epsilon_1\epsilon_2(v_1-v_0)}{2}t^2\Big),
\end{equation*}
for all $t\in[0,1/(\epsilon_1\epsilon_2)]$.
\end{remark}
With a stronger assumption on eigenvalues of the system, we obtain a better estimation of the error. A preliminary version of this result on three-level systems has been proposed in \cite{LiangBoscainSigalotti-CDC2024}.
\begin{proposition}
\label{prop:with_second_order}
    Assume that the assumptions of Theorem~\ref{theorem:main} are satisfied, and moreover, for all $1\leq j<k\leq n$ and $\alpha\in\mathcal{D}$, $\lambda_{k}(\alpha)-\lambda_{j}(\alpha)\notin[2v_0,2v_1]$. 
    Then there exist $C>0$ and $\eta>0$ such that for every $(\epsilon_1,\epsilon_2)\in(0,\eta)^2$,\begin{equation*}
\min_{\theta\in[0,2\pi]}\left\|\psi_{\epsilon_{1},\epsilon_{2}}\left(\frac{T}{\epsilon_{1}\epsilon_{2}}\right)-\exp(i\theta)e_{k}\right\|\leq C \max\left(\frac{\epsilon_{1}^{2}}{\epsilon_{2}},\frac{\epsilon_{2}}{\epsilon_{1}}\right).
\end{equation*}
\end{proposition}
{
\section{Sketch of the proof}\label{sec:sketch-proof}
The proof of Theorem~\ref{theorem:main} consists of the following steps:
\begin{itemize}
    \item Step 1: In equation \eqref{eq:interaction_frame}, the dynamics of \eqref{system:n_level} are recasted in the interaction frame. The dynamics in the interaction frame are characterized by oscillating terms of order $\mathcal{O}(\epsilon_1)$.
    
    \item Step 2: Hypothesis \ref{cond_2} of Theorem~\ref{theorem:main} allows us to introduce the change of variables in equation~\eqref{eq:first_cv} to eliminate all oscillating terms of order \( \mathcal{O}(\epsilon_1) \) except for the term with time-dependent frequency \( \lambda_{p}(\alpha) - \lambda_{q}(\alpha) + f(\epsilon_1 \epsilon_2 t) \) that couples the pair \( (\textbf{e}_{p}, \textbf{e}_{q}) \). Due to the non-linearity of the dynamics, this first-order elimination generates new oscillating terms of order \( \mathcal{O}(\epsilon_1^2) \).

\item Step 3: We show that, thanks to Hypothesis~\ref{cond_2} of Theorem~\ref{theorem:main}, the oscillating term with frequency \( \lambda_{p} - \lambda_{q} + f(\epsilon_1 \epsilon_2 t) \) introduces an implicit rotation between \( (\textbf{e}_{p}, \textbf{e}_{q}) \). This rotation is characterized by the radius \( \lambda_{\epsilon_1} \) and the angle \( \theta_{\epsilon_1} \), as defined in equations~\eqref{eq:def_lambda} and \eqref{eq:def_theta}. In equation~\eqref{eq:fourth_cv}, a change of variables is introduced to recast the dynamics in a rotational frame depending on \( \lambda_{\epsilon_1} \) and \( \theta_{\epsilon_1} \).
    
    \item Step 4: Using higher-order averaging results, we introduce the change of variables in equation~\eqref{eq:fifth_cv} to eliminate all oscillating terms of order $\mathcal{O}(\epsilon_1^2)$.

\item Step 5: In Lemma~\ref{lemma:RWA_1}, we obtain an approximate dynamics by truncating the residual term of the Hamiltonian. This truncation and the earlier steps correspond to the rotating wave approximation (RWA).

\item Step 6: We project the approximate dynamics onto the two-dimensional subspace \( \textbf{span}(\textbf{e}_{p}, \textbf{e}_{q}) \) and apply a non-standard adiabatic approximation to achieve population inversion between \( \textbf{e}_{p} \) and \( \textbf{e}_{q} \) in the approximate dynamics.

\item Step 7: Finally, we conclude the proof by combining the errors from the rotating wave approximation (Steps 1-5) and the adiabatic approximation (Step 6).
\end{itemize}
}
\section{Proof of Theorem~\ref{theorem:main}}\label{sec:proof}
\subsection*{Step 1: Interaction frame}
Define $\Delta(\alpha)=\lambda_{q}(\alpha)-\lambda_{p}(\alpha)$. In the following, we will replace $H(\alpha)$ by $H(\alpha)-\frac{1}{2}(\lambda_{p}(\alpha)+\lambda_{q}(\alpha))\mathbb{I}_{n}$, which will only introduce a relative phase to the state. Moreover, this transition will not change the gaps between eigenvalues, the system will always satisfy the conditions in Theorem \ref{theorem:main}, and $-\lambda_{p}(\alpha)=\lambda_{q}(\alpha)=\Delta(\alpha)/2.$
For $E\in \mathbb{R}$ and $1\leq j\leq k\leq n$, let us define 
\begin{equation}
\label{eq:def_A_j_k}
\begin{aligned}
    A_{jk}(E)&=\left\{\begin{matrix}
e^{iE}\textbf{e}_{jk}+e^{-iE}\textbf{e}_{kj} &\text{if }j<k,\\ 
\cos(E)\textbf{e}_{jj} & \text{if }j=k,
\end{matrix}\right.\\ B_{jk}(E)&=\left\{\begin{matrix}
\-ie^{iE}\textbf{e}_{jk}-ie^{-iE}\textbf{e}_{kj}&\text{if }j<k,\\ 
-\sin(E)\textbf{e}_{jj} & \text{if }j=k.
\end{matrix}\right.
\end{aligned}
\end{equation}
Let us fix $T=1$. For $\epsilon_1,\epsilon_2>0$, define
\begin{equation}
\label{eq:define_phi}
\phi:\left[0,\frac{1}{\epsilon_1\epsilon_2}\right]\ni t\mapsto \int_{0}^{t}f(\epsilon_1\epsilon_2\tau)\text{d}\tau.
\end{equation}
Let us  recast \eqref{system:n_level} in the interaction frame $\psi(t)=\exp(-itH(\alpha))\psi_{I}(t)$. Notice that 
\begin{equation}
\label{eq:interaction_frame}
    i\frac{\text{d}}{\text{d}t}\psi_{I}(t)=H_{I}(t)\psi(t),
\end{equation}
where
\begin{equation*}
\begin{aligned}
    H_{I}(t)=&-H(\alpha)+\exp(itH(\alpha))H(t)\exp(-itH(\alpha))\\
    =&\sum_{j=1}^{n}{\sum_{k=j}^{n}}\delta_{jk}\omega_{\epsilon_1,\epsilon_2}(t)A_{jk}\big((\lambda_{j}-\lambda_{k})t\big).
\end{aligned}
\end{equation*}
For {$(j,k)\in\{1,\dots,n\}^2$} and $\sigma\in\mathbb{Q}$, let us define
\begin{equation}
\begin{aligned}
\label{eq:def_f_phi}
    f_{jk}^{\sigma}&:[0,1]\ni s\mapsto \lambda_{j}-\lambda_{k}+\sigma f(s),\\
    \phi_{jk}^{\sigma}&:\left[0,\frac{1}{\epsilon_1\epsilon_2}\right]\ni t\mapsto (\lambda_{j}-\lambda_{k})t+\sigma\phi(t).
\end{aligned}
\end{equation}
Notice that by definition of $A_{jk}$ in equation \eqref{eq:def_A_j_k}, for all $E_1, E_2\in\mathbb{R}$, $2\cos(E_1)A_{jk}(E_2)=A_{jk}(E_2-E_1)+A_{jk}(E_2+E_1)$. Then, when applying the chirped pulse $\omega_{\epsilon_{1},\epsilon_{2}}(\cdot)$ given in \eqref{eq:control}, we obtain
\begin{equation}
\begin{aligned}
\label{eq:H_I}
    H_{I}(t)=\sum_{j=1}^{n}{\sum_{k=j}^{n}}\big(\epsilon_{1}\delta_{jk}u(\epsilon_{1}\epsilon_{2}t)A_{jk}\big(\phi_{jk}^{1}(t)\big)+\epsilon_{1}\delta_{jk}u(\epsilon_{1}\epsilon_{2}t)A_{jk}\big(\phi_{jk}^{-1}(t)\big)\big).
\end{aligned}
\end{equation}
\subsection*{Step 2: First-order elimination}
 In the following proposition, we give the expansion of the Hamiltonian after a general unitary change of variables.
\begin{proposition}[Change of variables]
\label{prop:change_of_var}
Consider the change of variables 
$\psi(t)=\exp(iX(t))\Tilde{\psi}(t)$,
 where $X(\cdot)$ is a smooth curve in the space of $n\times n$ Hermitian matrices and 
$\psi(\cdot)$ is the solution of the Schr\"odinger equation
\begin{equation*}
    i\frac{\text{d}}{\text{d}t}\psi(t)=H(t)\psi(t).
\end{equation*}
Then $\Tilde{\psi}(\cdot)$ is the solution of 
\begin{equation}
\label{eq:unitary_change_of_variables}
    i\frac{\text{d}}{\text{d}t}\Tilde{\psi}(t)=\Tilde{H}(t)\Tilde{\psi}(t),\quad\Tilde{\psi}(0)=\exp(-iX(0))\psi(0),
\end{equation}
where $\Tilde{H}(\cdot)$ is given by
\begin{equation*}
    \begin{aligned}
        \Tilde{H}(t)=\sum_{k=0}^{\infty}\frac{(-1)^{k}}{k!}{\rm ad}^{k}_{iX(t)}\left(H(t)+\frac{1}{k+1}\frac{\text{d}}{\text{d}t}X(t)\right).
    \end{aligned}
\end{equation*}
\end{proposition}
\begin{proof}
{
By differentiating $\psi(t)=\exp(iX(t))\Tilde{\psi}(t)$, we can deduce that the dynamics of $\tilde{\psi}$ is characterized by the Hamiltonian
\begin{equation*}
    \begin{aligned}
         \Tilde{H}(t)=\exp(-iX(t))H(t)\exp(iX(t))+\exp(-iX(t))\text{d}\exp_{iX(t)}\left(\frac{\text{d}}{\text{d}t}X(t)\right).
    \end{aligned}
\end{equation*}
By the Baker--Campbell--Hausdorff formula and by Theorem~4.5 in \cite{hall2000elementary}, we have
\begin{equation*}
    \begin{aligned}
        \exp(-iX(t))H(t)\exp(iX(t))=\sum_{k=1}^{\infty}\frac{(-1)^{k}}{k!}\text{ad}^{k}_{iX(t)}\left(H(t)\right),
    \end{aligned}
\end{equation*}
\begin{equation*}
\begin{aligned}
   \exp(-iX(t))\text{d}\exp_{iX(t)}\left(
    \frac{\text{d}}{\text{d}t}X(t)\right)=\sum_{k=0}^{\infty}\frac{(-1)^{k}}{(k+1)!}\text{ad}^{k}_{iX(t)}\left(\frac{\text{d}}{\text{d}t}X(t)\right).
\end{aligned}
\end{equation*}
Hence, we conclude that
\begin{equation*}
\begin{aligned}
    \Tilde{H}(t)=\sum_{k=0}^{\infty}\frac{(-1)^{k}}{k!}\text{ad}^{k}_{iX(t)}\left(H(t)+\frac{1}{k+1}\frac{\text{d}}{\text{d}t}X(t)\right).
\end{aligned}
\end{equation*}
}
\end{proof}
\begin{definition}\label{def:big_O}{  Fix $T>0$}, we call $R$ a \emph{$(\epsilon_{1},\epsilon_{2})$-parameterized function} if for every $\epsilon_{1},\epsilon_{2}>0$, $R_{\epsilon_{1},\epsilon_{2}}$ is a real-valued function defined on $\left[0,\frac{T
}{\epsilon_{1}\epsilon_{2}}\right]$. 
Given an $(\epsilon_{1},\epsilon_{2})$-parameterized function $R$ and  $g:\mathbb{R}_{+}^{2}\rightarrow\mathbb{R}_{+}$, we say that { $R_{\epsilon_1,\epsilon_2}(t)=\mathcal{O}\big(g(\epsilon_{1},\epsilon_{2})\big)$ over $\left[0,\frac{T}{\epsilon_1\epsilon_2}\right]$} if there exist $\delta,C>0$ such that for every $(\epsilon_{1},\epsilon_{2})\in(0,\delta)^{2}$ and $t\in\left[0,\frac{T}{\epsilon_{1}\epsilon_{2}}\right]$, we have $\left|R_{\epsilon_{1},\epsilon_{2}}(t)\right|\leq C g(\epsilon_{1},\epsilon_{2})$.
\end{definition}
Let us first define the  sets of indices
\begin{equation}
\begin{aligned}
    \mathcal{I}=\left\{(j,k,\sigma)\mid 1\leq j{\leq}k\leq n, \sigma=\pm 1\right\},\quad\mathcal{I}'=\mathcal{I}\setminus\left\{(p,q,1)\right\}.
\end{aligned}
\end{equation}
By Hypothesis~\ref{cond_2} of Theorem~\ref{theorem:main} and the assumptions made on $f$, for every $(j,k,\sigma)\in\mathcal{I}'$ and $s\in[0,1]$, $f_{jk}^{\sigma}(s)=\lambda_{j}-\lambda_{k}+\sigma f(s)\neq 0$, where $f_{jk}^{\sigma}(\cdot)$ is introduced in equation \eqref{eq:def_f_phi}. Then we can apply a first change of variables to  system \eqref{eq:interaction_frame} $\psi_{I}(t)=\exp(i\epsilon_{1}X_{1}(t))\hat{\psi}_{1}(t)$, where 
\begin{equation}
\label{eq:first_cv}
    X_{1}(t)=\sum_{(j,k,\sigma)\in\mathcal{I}'}\frac{\delta_{jk}u(\epsilon_{1}\epsilon_{2}t)}{f_{jk}(\epsilon_{1}\epsilon_{2}t)}B_{jk}\big(\phi_{jk}^{\sigma}(t)\big).
\end{equation}
Notice that, for every $(j,k,\sigma)\in\mathcal{I}'$, $\frac{\text{d}}{\text{d}t}B_{jk}\big(\phi_{jk}^{\sigma}(t)\big)=-f_{jk}^{\sigma}(\epsilon_{1}\epsilon_{2}t)A_{jk}\big(\phi_{jk}^{\sigma}(t)\big)$ and $\frac{\text{d}}{\text{d}t}\left(\frac{u(\epsilon_{1}\epsilon_{2}t)}{f_{jk}^{\sigma}(\epsilon_{1}\epsilon_{2}t)}\right)=\mathcal{O}(\epsilon_{1}\epsilon_{2})$
over $[0,1/(\epsilon_1\epsilon_2)]$, where $\mathcal{O}(\cdot)$ is defined in Definition \ref{def:big_O}. Then, by differentiating $X_{1}(t)$, we obtain that 
\begin{equation}
    \begin{aligned}
    \label{eq:d_dt_X_1}
     \frac{\text{d}}{\text{d}t}X_{1}(t)=-\sum_{(j,k,\sigma)\in\mathcal{I}'}\delta_{jk}u(\epsilon_{1}\epsilon_{2}t)A_{jk}\big(\phi_{jk}^{\sigma}(t)\big)+\mathcal{O}(\epsilon_{1}\epsilon_{2}).
    \end{aligned}
\end{equation}
By Proposition~\ref{prop:change_of_var}, we  deduce that the dynamics of $\hat{\psi}_{1}$ are characterized by the Hamiltonian
\begin{equation*}
\begin{aligned}
    \hat{H}_{1}(t)=H_{I}(t)+\epsilon_{1}\frac{\text{d}}{\text{d}t}X_1(t)-i\epsilon_{1}\left[X_1(t),H_{I}(t)+\epsilon_1\frac{1}{2}\frac{\text{d}}{\text{d}t}X_1(t)\right]+\mathcal{O}(\epsilon_{1}^{3}).
\end{aligned}
\end{equation*}
By equations \eqref{eq:H_I} and \eqref{eq:d_dt_X_1}, we obtain that
\begin{equation}
\begin{aligned}
\label{eq:commutator}
    \hat{H}_{1}(t)=\epsilon_{1}\delta_{pq}u(\epsilon_{1}\epsilon_{2}t)A_{pq}\big(\phi_{pq}^{1}(t)\big)-i\epsilon_{1}\left[X_1(t),H_{I}(t)+\epsilon_1\frac{1}{2}\frac{\text{d}}{\text{d}t}X_1(t)\right]+\mathcal{O}(\epsilon_{1}^{3}+\epsilon_{1}^{2}\epsilon_{2})
\end{aligned}
\end{equation}
over $\left[0,1/(\epsilon_1\epsilon_2)\right]$. For $s\in[0,1]$, let us define
\begin{equation*}
\begin{aligned}
c_{jk}^{\sigma}(s)&=\left\{
   \begin{matrix}\frac{\delta_{jk}u(s)}{f_{jk}^{\sigma}(s)} &\text{if }(j,k,\sigma)\in\mathcal{I}', \\ 0 &\text{if }(j,k,\sigma)=(p,q,1),\end{matrix}\right.\\
   l_{jk}^{\sigma}(s)&=\left\{\begin{matrix}
       \frac{1}{2}\delta_{jk}u(s) & \text{if }(j,k,\sigma)\in\mathcal{I}', \\
       \delta_{jk}u(s) & \text{if }(j,k,\sigma)=(p,q,1).
   \end{matrix}\right.
\end{aligned}
\end{equation*}
{Notice that when $j=k$, for every $s\in[0,1]$, we have $c_{jj}^{1}(s)=-c_{jj}^{-1}(s)=\delta_{jj}u(s)/f(s)$ and $l_{jj}^{1}(s)=l_{jj}^{-1}(s)=\delta_{jj}u(s)$.} For $j>k$ and $\sigma\in\{-1,1\}$ let us note, for every $s\in[0,1]$,
{\begin{equation*}
    \begin{aligned}
        c_{jk}^{\sigma}(s)&=-c_{kj}^{-\sigma}(s),\quad l_{jk}^{\sigma}(s)&=l_{kj}^{-\sigma}(s).
    \end{aligned}
\end{equation*}
Then, we obtain the following notations
\begin{equation*}
    \begin{aligned}
        -iX_{1}(t)&=\sum_{j=1}^{n}\sum_{k=1}^{n}\sum_{\sigma\in\{-1,1\}}c_{jk}^{\sigma}(\epsilon_1\epsilon_2t)e^{i\phi_{jk}^{\sigma}(t)}\textbf{e}_{jk}
    \end{aligned}
\end{equation*}
and
\begin{equation*}
\begin{aligned}
        H_{I}(t)+\epsilon_{1}\frac{1}{2}\frac{\text{d}}{\text{d}t}X_{1}(t)=\sum_{j=1}^{n}\sum_{k=1}^{n}\sum_{\sigma\in\{-1,1\}}l_{jk}^{\sigma}(\epsilon_1\epsilon_2t)e^{i\phi_{jk}^{\sigma}(t)}\textbf{e}_{jk}+\mathcal{O}(\epsilon_{1}^{2}\epsilon_{2}).
\end{aligned}
\end{equation*}
}
{ 
We can deduce from equation \eqref{eq:commutator} that 
\begin{equation}
\begin{aligned}
\label{eq:H_1}
    \Hat{H}_{1}(t)=\epsilon_{1}\delta_{pq}u(\epsilon_{1}\epsilon_{2}t)A_{pq}\big(\phi_{pq}^{1}(t)\big)+\sum_{(j,k,\sigma)\in\mathcal{J}}\epsilon_{1}^{2}h_{jk}^{\sigma}(\epsilon_{1}\epsilon_{2}t)A_{jk}\big(\phi_{jk}^{\sigma}(t)\big)+\mathcal{O}(\epsilon_{1}^{3}+\epsilon_{1}^{2}\epsilon_{1})
\end{aligned}
\end{equation}
over $[0,1/(\epsilon_1\epsilon_2)]$, where
}
\begin{equation}
\label{eq:def_J}
    \mathcal{J}=\left\{(j,k,\sigma)\mid 1\leq j\leq k\leq n,\sigma\in\{-2,0,2\}\right\}
\end{equation}
and, for every $1\leq j\leq k\leq n$ and $s\in[0,1]$,
\begin{equation*}
\begin{aligned}
h_{jk}^{2}(s)=&\sum_{m=1}^{n}c_{jm}^{1}(s)l_{mk}^{1}(s)-l_{jm}^{1}(s)c_{mk}^{1}(s),\\
h_{jk}^{0}(s)=&\sum_{m=1}^{n}\Big(c_{jm}^{-1}(s)l_{mk}^{1}(s)-l_{jm}^{1}(s)c_{mk}^{-1}(s)\Big)+\sum_{m=1}^{n}\Big(c_{jm}^{1}(s)l_{mk}^{-1}(s)-l_{jm}^{-1}(s)c_{mk}^{1}(s)\Big),\\
h_{jk}^{-2}(s)=&\sum_{m=1}^{n}c_{jm}^{-1}(s)l_{mk}^{-1}(s)-l_{jm}^{-1}(s)c_{mk}^{-1}(s).
\end{aligned}
\end{equation*}
\subsection*{Step 3: Implicit rotation $\theta$ between \( (\textbf{e}_{p}, \textbf{e}_{q}) \)} Set 
\begin{equation}
\begin{aligned}
    \mathcal{J}'=\{(j,k,\sigma)\mid 1\leq j<k\leq n,\sigma\in\{-2,0\}\}\cup\{(j,j,\sigma)\mid 1\leq j\leq n,\sigma=\pm 2\}.
\end{aligned}
\end{equation}
Since by the assumption of Theorem \ref{theorem:main}, we have $\lambda_{1}<\dots<\lambda_{n}$ and $f(s)>0$ for all $s\in[0,1]$, we can deduce that, for every $(j,k,\sigma)\in\mathcal{J}'$ and $s\in[0,1]$, $f_{jk}^{\sigma}(s)=\lambda_{j}-\lambda_{k}+\sigma f(s)\neq 0$. Let us introduce a second change of variables $\hat{\psi}_{1}(t)=\exp(i\epsilon_{1}^{2}X_{2}(t))\hat{\psi}_{2}(t)$, where
\begin{equation}
\label{eq:second_cv}
    X_{2}(t)=\sum_{(j,k,\sigma)\in\mathcal{J}'}\frac{h_{jk}^{\sigma}(\epsilon_{1}\epsilon_{2}t)}{f_{jk}^{\sigma}(\epsilon_{1}\epsilon_{2}t)}B_{jk}\big(\phi_{jk}^{\sigma}(t)\big).
\end{equation}
Notice that 
\[\frac{\text{d}}{\text{d}t}X_{2}(t)=-\sum_{(j,k,\sigma)\in\mathcal{J}'}h_{jk}^{\sigma}(\epsilon_{1}\epsilon_{2}t)A_{jk}\big(\phi_{jk}^{\sigma}(t)\big)+\mathcal{O}(\epsilon_{1}\epsilon_{2})\] over $[0,1/(\epsilon_1
\epsilon_2)]$. Then, by Proposition \ref{prop:change_of_var} and Equation \eqref{eq:H_1}, we deduce that, over $[0,1/(\epsilon_1\epsilon_2)]$, the dynamics of $\hat{\psi}_{2}(t)$ are characterized by the Hamiltonian
\begin{equation*}
\begin{aligned}
    \hat{H}_{2}(t)=&\hat{H}_{1}(t)+\epsilon_{1}^{2}\frac{\text{d}}{\text{d}t}X_{2}(t)+\mathcal{O}(\epsilon_{1}^{3})\\
=&\epsilon_1\delta_{pq}u(\epsilon_{1}\epsilon_{2}t)A_{pq}\big(\phi_{pq}^{1}(t)\big)+\sum_{j=1}^{n} \epsilon_1^2 h_{jj}^{0}(\epsilon_1\epsilon_2 t)\textbf{e}_{jj}
    \\&+\sum_{1\leq j<k\leq n}\epsilon_1^2h_{jk}^{2}(\epsilon_1\epsilon_2 t)A_{jk}(\phi_{jk}^{2}(t))+\mathcal{O}(\epsilon_1^{3}+\epsilon_1^{2}\epsilon_2).
\end{aligned}   
\end{equation*}
\begin{remark}\label{remark:final_state}
    Since $u(0)=u(1)=0$, we  deduce that $h_{jk}^{\sigma}(0)=h_{jk}^{\sigma}(1)=0$ for every $(j,k,\sigma)\in\mathcal{J}$.
    Then $X_{1}(0)=X_{1}\left(\frac{1}{\epsilon_1\epsilon_2}\right)=0$ and $X_{2}(0)=X_{2}\left(\frac{1}{\epsilon_1\epsilon_2}\right)=0$. Hence, $\psi_{I}(0)=\hat{\psi}_2(0)$ and $\psi_{I}\left(\frac{1}{\epsilon_1\epsilon_2}\right)=\hat{\psi}_2\left(\frac{1}{\epsilon_2\epsilon_2}\right)$.
\end{remark}
 Notice that $-\lambda_{p}=\lambda_{q}=\Delta/2$. Let us introduce a third change of variables, namely, $\hat{\psi}_{2}(t)=U_{3}^{\dagger}(t)\hat{\psi}_{3}(t)$, where
\begin{equation}
\label{eq:third_cv}
\begin{aligned}
    U_{3}(t)=\sum_{j\notin\{p,q\}}\textbf{e}_{jj}+\exp\left(\frac{i}{2}\int_{0}^{t}\Delta-f(\epsilon_1\epsilon_2\tau)\text{d}\tau\right)\textbf{e}_{pp}+\exp\left(-\frac{i}{2}\int_{0}^{t}\Delta-f(\epsilon_1\epsilon_2\tau)\text{d}\tau\right)\textbf{e}_{qq}.
\end{aligned}
\end{equation}
Let us define
\begin{equation}
    \mathcal{K}=\left\{(j,k)\mid 1\leq j<k\leq n \text{ and }j,k\notin\{p,q\}\right\}.
\end{equation}
Then the dynamics of $\hat{\psi}_{3}$ are characterized by the Hamiltonian
\begin{equation*}
\begin{aligned}
    \hat{H}_{3}(t)=&\frac{1}{2}(\Delta-f(\epsilon_1\epsilon_2 t))(\textbf{e}_{qq}-\textbf{e}_{pp})+\epsilon_1\delta_{pq}u(\epsilon_1\epsilon_2 t)(\textbf{e}_{pq}+\textbf{e}_{qp})+\sum_{j=1}^{n} \epsilon_1^2 h_{jj}^{0}(\epsilon_1\epsilon_2 t)\textbf{e}_{jj}\\&+\epsilon_1^2h_{pq}^{2}(\epsilon_1\epsilon_2 t)A_{pq}(\phi(t))+\epsilon_1^2R(t)+\epsilon_1^2\tilde{R}(t)+\mathcal{O}(\epsilon_1^3+\epsilon_1^2\epsilon_2),
\end{aligned}
\end{equation*}
where
\begin{equation}
\begin{aligned}
R(t)=\sum_{(j,k)\in\mathcal{K}}h_{jk}^{2}(\epsilon_1\epsilon_2 t)A_{jk}(\phi_{jk}^{2}(t))
\end{aligned}
\end{equation}
and
\begin{equation}
\begin{aligned}
\label{eq:def_R}
        \tilde{R}(t)=&\sum_{j=1}^{p-1}\Big(h^{2}_{jp}(\epsilon_1\epsilon_2 t)A_{jp}\left(\lambda_{j}t+\frac{5\phi(t)}{2}\right)+h^{2}_{jq}(\epsilon_1\epsilon_2 t)A_{jq}\left(\lambda_{j}t+\frac{3\phi(t)}{2}\right)\Big)\\
    &+\sum_{j=p+1}^{q-1}\Big(h^2_{pj}(\epsilon_1\epsilon_2 t)A_{pj}\left(-\lambda_{j}t+\frac{3\phi(t)}{2}\right)+h_{jq}^2(\epsilon_1\epsilon_2 t)A_{jq}\left(\lambda_{j}t+\frac{3\phi(t)}{2}\right)\Big)\\
    &+\sum_{j=q+1}^{n}\Big(h_{pj}^{2}(\epsilon_1\epsilon_2 t)A_{pj}\left(-\lambda_j t+\frac{3\phi(t)}{2}\right)+h_{qj}^{2}(\epsilon_1\epsilon_2 t)A_{qj}\left(-\lambda_{j}t+\frac{5\phi(t)}{2}\right)\Big).
\end{aligned}
\end{equation}
For $\epsilon_1,\epsilon_2>0$ and $s\in[0,1]$, let us define
\begin{align}
    \lambda_{\epsilon_1}(s)&=\sqrt{\frac{1}{4}(\Delta-f(s))^2+\epsilon_1^2\delta_{pq}^2 u(s)^2},\label{eq:def_lambda}\\
    \theta_{\epsilon_1}(s)&=\textbf{sgn}(\delta_{pq})\arccos\left(\frac{\Delta-f(s)}{2\lambda_{\epsilon_1}(s)}\right)\label{eq:def_theta},
\end{align}
where $\textbf{sgn}(\delta_{pq})$ is the sign of $\delta_{pq}$. 
\begin{remark}
 By definition of   $\lambda_{\epsilon_1}$ and $\theta_{\epsilon_1}$, it follows that, for $s\in[0,1]$, 
    \begin{equation*}
    \begin{aligned}
        \begin{pmatrix}
        -\frac{1}{2}(\Delta-f(s)) & \epsilon_1\delta_{pq}u(s)\\
        \epsilon_1\delta_{pq}u(s) & \frac{1}{2}(\Delta-f(s))
        \end{pmatrix}=\lambda_{\epsilon_1}(s)\begin{pmatrix}
            -\cos\big(\theta_{\epsilon_1}(s)\big) & \sin\big(\theta_{\epsilon_1}(s)\big)\\
            \sin\big(\theta_{\epsilon_1}(s)\big) & \cos\big(\theta_{\epsilon_1}(s)\big)
        \end{pmatrix}.
    \end{aligned}
    \end{equation*}
Moreover, the assumptions on $f$ and $u$ given in Theorem~\ref{theorem:main} ensure that $\theta_{\epsilon_1}(0)=0$ and $\theta_{\epsilon_1}(1)=\textbf{sgn}(\delta_{pq})\pi$.
\end{remark}
\begin{lemma}
\label{lemma:theta_derivative_order}
$\|\dot{\theta}_{\epsilon_1}\|_{\infty}=\mathcal{O}(1/\epsilon_1)$.
\end{lemma}
\begin{proof}
    By the definitions \eqref{eq:def_lambda} and \eqref{eq:def_theta}, we can deduce that, for all $s\in[0,1]$,
    \begin{equation*}
        |\dot{\theta}_{\epsilon_1}(s)|\leq\frac{\sqrt{\frac{1}{4}\dot{f}(s)^2+\epsilon_1^2\dot{u}(s)^2}}{\lambda_{\epsilon_1}(s)}\leq\frac{\frac{1}{2}\|\dot{f}\|_{\infty}+\epsilon_1\|\dot{u}\|_{\infty}}{\lambda_{\epsilon_1}(s)}.
    \end{equation*}
      Since $\|1/\lambda_{\epsilon_1}\|_{\infty}=\mathcal{O}(1/\epsilon_1)$, this concludes the proof.
\end{proof}
Since $f$ in \eqref{eq:control} is strictly increasing on $[0,1]$, let us define the function $\bar{s}:\mathcal{D}\rightarrow(0,1)$ such that
\begin{equation}
\label{eq:def_s_bar}
     s=\bar{s}(\alpha)\iff \Delta(\alpha)=f(\bar{s}(\alpha)).
\end{equation}
\begin{lemma}
\label{lemma:int_theta}
    There exist $M>0$ and $\eta>0$ such that, for all $\epsilon_1\in(0,\eta)$, $ \int_{0}^{1}\left|\dot{\theta}_{\epsilon_1}(s)\right|\text{d}s<M$.
\end{lemma}
The proof of the lemma can be found in the appendix.
\begin{remark}
In the following, for simplicity of notation, we will denote $\theta_{\epsilon_1}(\epsilon_1\epsilon_2 t)$ and $\dot{\theta}_{\epsilon_1}(\epsilon_1\epsilon_2 t)$ simply as $\theta_{\epsilon_1}$ and $\dot{\theta}_{\epsilon_1}$. Notice that $\frac{\text{d}}{\text{d}t}\theta_{\epsilon_1}=\epsilon_1\epsilon_2\dot{\theta}_{\epsilon_1}$.
\end{remark}
\subsection*{Step 4: Second-order elimination}
Let us define $ \tilde{\phi}(t)=\int_{0}^{t}{\lambda}_{\epsilon_1}(\epsilon_1\epsilon_2\tau)\text{d}\tau$,
 and introduce the change of variables $\hat{\psi}_{3}(t)=U_{4}^{\dagger}(t)\hat{\psi}_{4}(t)$, where
\begin{equation}
\begin{aligned}
\label{eq:fourth_cv}
    U_{4}(t)=& \sin\left(\frac{\theta_{\epsilon_1}}{2}\right)\Big(-\exp(-i\tilde{\phi}(t))\textbf{e}_{pq}+\exp(i\tilde{\phi}(t))\textbf{e}_{qp}\Big)\\
&+\cos\left(\frac{\theta_{\epsilon_1}}{2}\right)\Big(\exp(-i\tilde{\phi}(t))\textbf{e}_{pp}+\exp(i\tilde{\phi}(t))\textbf{e}_{qq}\Big)+\sum_{j\notin\{p,q\}}\textbf{e}_{jj}.
\end{aligned}
\end{equation}
The dynamics of $\hat{\psi}_{4}$ are characterized by the Hamiltonian
\begin{equation}
\begin{aligned}
\label{eq:def_hat_H_4}
    \hat{H}_{4}(t)=&-\frac{\epsilon_1\epsilon_2}{2}\dot{\theta}_{\epsilon_1}B_{pq}(-2\tilde{\phi}(t))+\epsilon_1^2 U_4(t)\Big(\sum_{j=1}^{n}h_{jj}^{0}(\epsilon_1\epsilon_2 t)\textbf{e}_{jj}\Big)U_{4}^{\dagger}(t)\\
    &+\epsilon_1^2R(t)+\epsilon_1^2\tilde{R}_{pq}(t)+\epsilon_1^2\tilde{R}_{p}(t)+\epsilon_1^2\tilde{R}_{q}(t)+\mathcal{O}(\epsilon_1^3+\epsilon_1^2\epsilon_2).
\end{aligned}
\end{equation}
Here $R(t)$ is given in equation \eqref{eq:def_R}, $\tilde{R}_{pq}(t)$ is given by
\begin{equation}
      \begin{aligned}
         \tilde{R}_{pq}(t)=&-\sin(\theta_{\epsilon_1})h_{pq}^{2}(\epsilon_1\epsilon_2 t)A_{pp}\left(\phi(t)\right)+\sin(\theta_{\epsilon_1})h_{pq}^{2}(\epsilon_1\epsilon_2 t)A_{qq}\left(\phi(t)\right)\\
    &+\cos^2\left(\frac{\theta_{\epsilon_1}}{2}\right)h_{pq}^{2}(\epsilon_1\epsilon_2 t)A_{pq}\left(\tilde{\phi}_{pq}^{c}(t)\right)-\sin^2\left(\frac{\theta_{\epsilon_1}}{2}\right)h_{pq}^{2}(\epsilon_1\epsilon_2 t)A_{pq}\left(\tilde{\phi}_{pq}^{s}(t)\right),
    \end{aligned}
\end{equation}
where $\phi(t)$ is defined in equation \eqref{eq:define_phi} and
\begin{equation}
\begin{aligned}
\label{eq:def_tilde_phi_1}
    \tilde{\phi}_{pq}^{c}(t)&=-2\tilde{\phi}(t)+\phi(t),\ 
    \tilde{\phi}_{pq}^{s}(t)=-2\tilde{\phi}(t)-\phi(t),
\end{aligned}
\end{equation}
$\tilde{R}_{p}(t)$ is given by
\begin{equation}
\begin{aligned}
    \tilde{R}_{p}(t)=&\sum_{j=1}^{p-1}\Big(\cos\left(\frac{\theta_{\epsilon_1}}{2}\right)h_{jp}^{2}(\epsilon_1\epsilon_2 t)A_{jp}(\tilde{\phi}_{jp}^{c}(t))-\sin\left(\frac{\theta_{\epsilon_1}}{2}\right)h_{jq}^{2}(\epsilon_1\epsilon_2 t)A_{jp}(\tilde{\phi}_{jp}^{s}(t))\Big)\\
&+\sum_{j=p+1}^{q-1}\Big(\cos\left(\frac{\theta_{\epsilon_1}}{2}\right)h_{pj}^2(\epsilon_1\epsilon_2 t)A_{pj}(\tilde{\phi}_{pj}^{c}(t))-\sin\left(\frac{\theta_{\epsilon_1}}{2}\right)h_{jq}^{2}(\epsilon_1\epsilon_2 t)A_{pj}(\tilde{\phi}_{pj}^{s}(t))\Big)\\
&+\sum_{j=q+1}^{n}\Big(\cos\left(\frac{\theta_{\epsilon_1}}{2}\right)h_{pj}^{2}(\epsilon_1\epsilon_2 t)A_{pj}(\tilde{\phi}_{pj}^{c}(t))-\sin\left(\frac{\theta_{\epsilon_1}}{2}\right)h_{qj}^{2}(\epsilon_1\epsilon_2 t)A_{pj}(\tilde{\phi}_{pj}^{s}(t))\Big),
\end{aligned}
\end{equation}
where
\begin{equation}
\label{eq:def_phi_R_p}
    \begin{aligned}
        &\left\{\begin{matrix}
        \tilde{\phi}_{jp}^{c}(t)=\lambda_j t+\tilde{\phi}(t)+\frac{5\phi(t)}{2},\\
        \tilde{\phi}_{jp}^{s}(t)=\lambda_j t+\tilde{\phi}(t)+\frac{3\phi(t)}{2},
        \end{matrix}\ \text{ if }1\leq j<p\right.;\\
        &\left\{\begin{matrix}
            \tilde{\phi}_{pj}^{c}(t)=-\tilde{\phi}(t)-\lambda_j t+\frac{3\phi(t)}{2},\\
            \tilde{\phi}_{pj}^{s}(t)=-\tilde{\phi}(t)-\lambda_j t -\frac{3\phi(t)}{2},
        \end{matrix}\ \text{ if }p<j<q\right.;\\
        &\left\{\begin{matrix}
            \tilde{\phi}_{pj}^{c}(t)=-\tilde{\phi}(t)-\lambda_j t+\frac{3\phi(t)}{2},\\
            \tilde{\phi}_{pj}^{s}(t)=-\tilde{\phi}(t)-\lambda_j t+\frac{5\phi(t)}{2},
        \end{matrix}\ \text{ if }q<j\leq n\right. ,
    \end{aligned}
\end{equation}
and, finally, $\tilde{R}_{q}(t)$ is given by
\begin{equation}
    \begin{aligned}
        \tilde{R}_{q}(t)=&\sum_{j=1}^{p-1}\Big(\sin\left(\frac{\theta_{\epsilon_1}}{2}\right)h_{jp}^{2}(\epsilon_1\epsilon_2 t)A_{jq}(\tilde{\phi}_{jq}^{s}(t))+\cos\left(\frac{\theta_{\epsilon_1}}{2}\right)h_{jq}^{2}(\epsilon_1\epsilon_2 t)A_{jq}(\tilde{\phi}_{jq}^{c}(t))\Big)\\
        &+\sum_{j=p+1}^{q-1}\Big(\sin\left(\frac{\theta_{\epsilon_1}}{2}\right)h_{pj}^{2}(\epsilon_1\epsilon_2 t)A_{jq}(\tilde{\phi}_{jq}^{s}(t))+\cos\left(\frac{\theta_{\epsilon_1}}{2}\right)h_{jq}^2(\epsilon_1\epsilon_2 t)A_{jq}(\tilde{\phi}_{jq}^{c}(t))\Big)\\
        &+\sum_{j=q+1}^{n}\Big(\sin\left(\frac{\theta_{\epsilon_1}}{2}\right)h_{pj}^{2}(\epsilon_1\epsilon_2 t)A_{qj}(\tilde{\phi}_{qj}^{s}(t))+\cos\left(\frac{\theta_{\epsilon_1}}{2}\right)h_{qj}^{2}(\epsilon_1\epsilon_2 t)A_{qj}(\tilde{\phi}_{qj}^{c}(t))\Big),
    \end{aligned}
\end{equation}
where
\begin{equation}
\label{eq:def_phi_R_q}
    \begin{aligned}
        &\left\{\begin{matrix}
        \tilde{\phi}_{jq}^{c}(t)=\lambda_j t-\tilde{\phi}(t)+\frac{3\phi(t)}{2},\\
        \tilde{\phi}_{jq}^{s}(t)=\lambda_j t-\tilde{\phi}(t)+\frac{5\phi(t)}{2},
        \end{matrix}\ \text{ if }1\leq j<p\right.;\\
        &\left\{\begin{matrix}
            \tilde{\phi}_{jq}^{c}(t)=\lambda_j t-\tilde{\phi}(t)+\frac{3\phi(t)}{2},\\
            \tilde{\phi}_{jq}^{s}(t)=\lambda_j t-\tilde{\phi}(t) -\frac{3\phi(t)}{2},
        \end{matrix}\ \text{ if }p<j<q\right.;\\
        &\left\{\begin{matrix}
            \tilde{\phi}_{qj}^{c}(t)=\tilde{\phi}(t)-\lambda_j t+\frac{5\phi(t)}{2},\\
            \tilde{\phi}_{qj}^{s}(t)=\tilde{\phi}(t)-\lambda_j t+\frac{3\phi(t)}{2},
        \end{matrix}\ \text{ if }q<j\leq n\right. .
    \end{aligned}
\end{equation}
The proof of the following three lemmas can be found in the appendix.
\begin{lemma}
\label{lemma:function_integral_1}
    For $\epsilon_1,\epsilon_2>0$, given two smooth functions $g:[-\pi,\pi]\rightarrow\mathbb{R}$ and $h:[0,1]\rightarrow\mathbb{R}$, define $F_{\epsilon_1}(s)=g(\theta_{\epsilon_1}(s))h(s)$, where $\theta$ is defined as in \eqref{eq:def_theta}. Then there exist $c>0$ and $\eta>0$ such that for all $\epsilon_1\in(0,\eta)$ and $s\in[0,1]$, $\int_{0}^{s} \left|\dot{F}_{\epsilon_1}(\xi)\right|\text{d}\xi<c$.
\end{lemma}
\begin{lemma}
\label{lemma:van_der_corput_1}
    There exists $\eta>0$ such that for all $\epsilon_1\in(0,\eta)$, $\alpha\in\mathcal{D}$, and $s\in[0,1]$, we have
    \begin{equation*}
        \begin{aligned}
            -2\dot{\lambda}_{\epsilon_1}(s)+\dot{f}(s)&>\frac{1}{2}\dot{f}(s),\text{ if } s\in[0,\bar{s}(\alpha)],\\
            -2\lambda_{\epsilon_1}(s)+f(s)&>\frac{1}{2}\Delta(\alpha)>\frac{v_0}{2},\text{ if } s\in[\bar{s}(\alpha),1].
        \end{aligned}
    \end{equation*}
    Here $\bar{s}(\alpha)$ is defined as in equation \eqref{eq:def_s_bar}.
\end{lemma}
\begin{lemma}
    \label{lemma:van_der_corput_2}  There exists $\eta>0$ such that, for all $\epsilon_1\in(0,\eta)$ and $s\in[0,1]$, $ \left|\dot{\lambda}_{\epsilon_1}(s)\right|<|\dot{f}(s)|$.
\end{lemma}

In the proofs of the following lemmas, we are going to apply the adaptations of the classical Van der Corput lemma (see for example \cite{stein1993harmonic}) to quantum systems proposed in \cite{AugierMCRF}.
\begin{lemma}
\label{lemma:oscillatory_int_1}
$\int_{0}^{t}\tilde{R}_{pq}(\tau)\text{d}\tau=\mathcal{O}\left(\frac{1}{\sqrt{\epsilon_1\epsilon_2}}\right)$ over $[0,1/(\epsilon_1\epsilon_2)]$.
\end{lemma}
\begin{proof}
      Take $t\in[0,1/(\epsilon_1\epsilon_2)]$. By the change of variables $s=\epsilon_1\epsilon_2 \tau$, we have
    \begin{equation}
    \label{eq:int_function_1}
        \begin{aligned}
            &\int_0^{t}\sin(\theta_{\epsilon_1}(\epsilon_1\epsilon_2\tau))h_{pq}^2(\epsilon_1\epsilon_2\tau)e^{i\phi(\tau)}\text{d}\tau=\frac{1}{\epsilon_1\epsilon_2}\int_{0}^{\epsilon_1\epsilon_2\tau}\sin(\theta_{\epsilon_1}(s))h_{pq}^{2}(s)e^{i\phi(s/(\epsilon_1\epsilon_2))}\text{d}s.
        \end{aligned}
    \end{equation}
     By Lemma~\ref{lemma:function_integral_1}, there exist $\eta>0$ and $c_1>0$ such that for all $\epsilon_1,\epsilon_2\in(0,\eta)$ and $t\in[0,1/(\epsilon_1\epsilon_2)]$, we have
     $\int_{0}^{\epsilon_1\epsilon_2 t}\left|\frac{\text{d}}{\text{d}s}\big(\sin(\theta_{\epsilon_1}(s))h_{pq}^{2}(s)\big)\right|\text{d}s< c_1$. Notice that, by equation \eqref{eq:define_phi} and hypotheses of Theorem \ref{theorem:main}, we have $\frac{\text{d}}{\text{d}s}\phi\left(\frac{s}{\epsilon_1\epsilon_2}\right)=\frac{f(s)}{\epsilon_1\epsilon_2}\geq\frac{v_0}{\epsilon_1\epsilon_2}$. Then by equation~\eqref{eq:int_function_1} and Corollary A.7 in \cite{AugierMCRF}, we can deduce that there exists $c>0$ such that, for all $\epsilon_1,\epsilon_2\in(0,\eta)$ and $t\in[0,1/(\epsilon_1\epsilon_2)]$,
    \begin{equation*}
    \begin{aligned}
         \left|\frac{1}{\epsilon_1\epsilon_2}\right.&\left.\int_0^{\epsilon_1\epsilon_2 t}\sin(\theta_{\epsilon_1}(s))h_{pq}^2(s)e^{i\phi(s/(\epsilon_1\epsilon_2))}\text{d}\tau\right|\\
         \leq&\frac{c}{\epsilon_1\epsilon_2}\frac{\epsilon_1\epsilon_2}{v_0}\Bigg(|\sin(\theta_{\epsilon_1}(\epsilon_1\epsilon_2 t))h_{pq}^{2}(\epsilon_1\epsilon_2 t)|+\int_{0}^{\epsilon_1\epsilon_2 t}\left|\frac{\text{d}}{\text{d}s}\big(\sin(\theta_{\epsilon_1}(s))h_{pq}^{2}(s)\big)\right|\text{d}s\Bigg)<\frac{c}{v_0}(c_1+\|h_{pq}^{2}\|_{\infty}).
    \end{aligned}
    \end{equation*}
      Then by equation \eqref{eq:int_function_1} and the definition of $A_{jj}$ in equation \eqref{eq:def_A_j_k}, we obtain that, for all $j\in\{p,q\}$,
\begin{equation}
\label{eq:int_1}
    \int_{0}^{t}\sin(\theta_{\epsilon_1}(\epsilon_1\epsilon_2 \tau))h_{pq}^{2}(\epsilon_1\epsilon_2 \tau)A_{jj}(\phi(\tau))\text{d}\tau=\mathcal{O}(1)
\end{equation}
{  over $[0,1/(\epsilon_1\epsilon_2)].$}
Take $t\in[0,\bar{s}(\alpha)/(\epsilon_1\epsilon_2)]$. By the change of variables $s=\epsilon_1\epsilon_2 \tau$,
\begin{equation}
\begin{aligned}
\label{eq:int_function_2}
    &\int_0^{t}\cos^{2}\left(\frac{\theta_{\epsilon_1}(\epsilon_1\epsilon_2\tau)}{2}\right)h_{pq}^{2}(\epsilon_1\epsilon_2\tau)e^{i\tilde{\phi}_{pq}^{c}(\tau)}\text{d}\tau=\frac{1}{\epsilon_1\epsilon_2}\int_0^{\epsilon_1\epsilon_2 t}\cos^{2}\left(\frac{\theta_{\epsilon_1}(s)}{2}\right)h_{pq}^{2}(s)e^{i\tilde{\phi}_{pq}^{c}(s/(\epsilon_1\epsilon_2))}\text{d}s.
\end{aligned}
\end{equation}
Notice that $\frac{\text{d}^2}{\text{d}s^2}\tilde{\phi}_{pq}^c\left(\frac{s}{\epsilon_1\epsilon_2}\right)=\frac{1}{\epsilon_1\epsilon_2}(-2\dot{\lambda}_{\epsilon_1}(s)+\dot{f}(s))>\frac{1}{2}\dot{f}(s)$ by Lemma \ref{lemma:van_der_corput_1}. By the assumptions of Theorem \ref{theorem:main}, $\dot{f}(s)\ge \min(\dot{f})>0$ for all $s\in[0,\bar{s}(\alpha)]$. By applying Corollary A.6 in \cite{AugierMCRF} with $k=2$ and Lemma \ref{lemma:function_integral_1} to equation \eqref{eq:int_function_2}, we can prove that, { over $[0,\bar{s}(\alpha)/(\epsilon_1\epsilon_2)]$},
\begin{equation*}
    \int_0^{t}\cos^{2}\left(\frac{\theta_{\epsilon_1}(\epsilon_1\epsilon_2\tau)}{2}\right)h_{pq}^{2}(\epsilon_1\epsilon_2\tau)e^{i\tilde{\phi}_{pq}^{c}(\tau)}\text{d}\tau=\mathcal{O}\left(\frac{1}{\sqrt{\epsilon_1\epsilon_2}}\right).
\end{equation*}
{ Notice that, by Lemma \ref{lemma:van_der_corput_1}, there exists $\eta'>0$ such that, for all $\epsilon_1,\epsilon_2\in(0,\eta')$, we have $\frac{\text{d}}{\text{d}s}\tilde{\phi}_{pq}^{c}\left(\frac{s}{\epsilon_1\epsilon_2}\right)=\frac{1}{\epsilon_1\epsilon_2}(-2\lambda_{\epsilon_1}(s)+f(s))>\frac{v_0}{2\epsilon_1\epsilon_2}>0$.}
Then by a reasoning similar to that above, we can prove that, { over $[\bar{s}(\alpha)/(\epsilon_1\epsilon_2),1/(\epsilon_1\epsilon_2)]$}, 
    \begin{equation*}
          \int_{\frac{\bar{s}(\alpha)}{(\epsilon_1\epsilon_2)}}^{t}\cos^{2}\left(\frac{\theta_{\epsilon_1}(\epsilon_1\epsilon_2\tau)}{2}\right)h_{pq}^{2}(\epsilon_1\epsilon_2\tau)e^{i\tilde{\phi}_{pq}^{c}(\tau)}\text{d}\tau=\mathcal{O}\left(1\right).
    \end{equation*}
Then by the definition of $A_{pq}$ in equation \eqref{eq:def_A_j_k}, we have that, { over $[0,1/(\epsilon_1\epsilon_2)]$,}
\begin{equation}
\label{eq:int_2}
\begin{aligned}
     \int_0^{t}\cos^{2}\left(\frac{\theta_{\epsilon_1}(\epsilon_1\epsilon_2\tau)}{2}\right)h_{pq}^{2}(\epsilon_1\epsilon_2\tau)A_{pq}(\tilde{\phi}_{pq}^{c}(\tau))\text{d}\tau=\mathcal{O}\left(\frac{1}{\sqrt{\epsilon_1\epsilon_2}}\right).
\end{aligned}
\end{equation}
By definition of $\tilde{\phi}_{pq}^{s}$ in equation \eqref{eq:def_tilde_phi_1}, we deduce that $\frac{\text{d}}{\text{d}s}\tilde{\phi}_{pq}^{s}\left(\frac{s}{\epsilon_1\epsilon_2}\right)=\frac{1}{\epsilon_1\epsilon_2}(-2\lambda_{\epsilon_1}(s)-f(s))<-\frac{v_0}{\epsilon_1\epsilon_2}$. Similarly, we can obtain that, { over $t\in[0,1/(\epsilon_1\epsilon_2)]$},
\begin{equation}
\label{eq:int_3}
\begin{aligned}
     \int_0^{t}\cos^{2}\left(\frac{\theta_{\epsilon_1}(\epsilon_1\epsilon_2\tau)}{2}\right)h_{pq}^{2}(\epsilon_1\epsilon_2\tau)A_{pq}(\tilde{\phi}_{pq}^{s}(\tau))\text{d}\tau=\mathcal{O}\left(1\right).
\end{aligned}
\end{equation}
By equations \eqref{eq:int_1}, \eqref{eq:int_2}, and \eqref{eq:int_3}, we can conclude the proof.
\end{proof}
\begin{lemma}
\label{lemma:oscillatory_int_2}
    Over $t[0,1/(\epsilon_1\epsilon_2)]$,
    \begin{equation*}
        \begin{aligned}
            \int_0^t R(\tau)\text{d}\tau=\mathcal{O}\left(\frac{1}{\sqrt{\epsilon_1\epsilon_2}}\right),\quad\int_0^t \tilde{R}_{p}(\tau)\text{d}\tau=\mathcal{O}\left(\frac{1}{\sqrt{\epsilon_1\epsilon_2}}\right),\quad\int_0^t \tilde{R}_{q}(\tau)\text{d}\tau=\mathcal{O}\left(\frac{1}{\sqrt{\epsilon_1\epsilon_2}}\right).
        \end{aligned}
    \end{equation*}
\end{lemma}
\begin{proof}
    Notice that $\frac{\text{d}^2}{\text{d}t^2}\phi_{jk}^{2}\left(\frac{s}{\epsilon_1\epsilon_2}\right)=\frac{2}{\epsilon_1\epsilon_2}\dot{f}(s)>\frac{2}{\epsilon_1\epsilon_2}\min\dot{f}>0$ for all $s\in[0,1]$ and $(j,k)\in\mathcal{K}$. Then by the same reasoning as that in the proof of Lemma \ref{lemma:oscillatory_int_1}, we obtain that, 
    $ \int_{0}^{t}R(\tau)\text{d}\tau=\mathcal{O}(1/\sqrt{\epsilon_1\epsilon_2})$ { over $[0,1/(\epsilon_1\epsilon_2)]$}.

    Take $j$ such that $1\leq j<p$. By equation \eqref{eq:def_phi_R_p} and Lemma \ref{lemma:van_der_corput_2}, { there exists $\eta>0$ such that, for all $\epsilon_1,\epsilon_2\in(0,\eta)$ and $s\in[0,1]$,}
    \begin{equation*}
    \begin{aligned}
        \left|\frac{\text{d}^2}{\text{d}s^2}\tilde{\phi}_{jp}^{c}(s/(\epsilon_1\epsilon_2))\right|=&\frac{1}{\epsilon_1\epsilon_2}|\dot{\lambda}_{\epsilon_1}(s)+5\dot{f}(s)/2|>\frac{3}{2\epsilon_1\epsilon_2}\min \dot{f}>0,\\
        \left|\frac{\text{d}^2}{\text{d}s^2}\tilde{\phi}_{jp}^{s}(s/(\epsilon_1\epsilon_2))\right|=&\frac{1}{\epsilon_1\epsilon_2}|\dot{\lambda}_{\epsilon_1}(s)+3\dot{f}(s)/2|>\frac{1}{2\epsilon_1\epsilon_2}\min \dot{f}>0.
    \end{aligned}
    \end{equation*}
    Similarly, for $j$ such that $p<j\leq n$, {  $\epsilon_1,\epsilon_2\in(0,\eta)$ and $s\in[0,1]$,}
     \begin{equation*}
    \begin{aligned}
        \left|\frac{\text{d}^2}{\text{d}s^2}\tilde{\phi}_{pj}^{c}(s/(\epsilon_1\epsilon_2))\right|>\frac{1}{2\epsilon_1\epsilon_2}\min \dot{f}>0,\\
        \left|\frac{\text{d}^2}{\text{d}s^2}\tilde{\phi}_{pj}^{s}(s/(\epsilon_1\epsilon_2))\right|>\frac{1}{2\epsilon_1\epsilon_2}\min \dot{f}>0.
    \end{aligned}
    \end{equation*}
    By the same reasoning as in the proof of Lemma \ref{lemma:oscillatory_int_1}, we can prove that, $ \int_{0}^{t}\tilde{R}_{p}(\tau)\text{d}\tau=\mathcal{O}\left(\frac{1}{\sqrt{\epsilon_1\epsilon_2}}\right)$ { over  $[0,1/(\epsilon_1\epsilon_2)]$}.
   Similarly, by equation \eqref{eq:def_phi_R_q} and Lemma~\ref{lemma:van_der_corput_2}, we can obtain that, $  \int_{0}^{t}\tilde{R}_{q}(\tau)\text{d}\tau=\mathcal{O}\left(\frac{1}{\sqrt{\epsilon_1\epsilon_2}}\right)$ { over  $[0,1/(\epsilon_1\epsilon_2)]$}.
\end{proof}
Let us introduce a fifth change of variables
\begin{equation}
\label{eq:fifth_cv}
\hat{\psi}_4(t)=\exp(i\epsilon_1^2X_5(t))\hat{\psi}_5(t),
\end{equation}
where 
\begin{equation}
\label{eq:def_X_5}
    X_5(t)=-\int_{0}^{t}(R(\tau)+\tilde{R}_{pq}(\tau)+\tilde{R}_{p}(\tau)+\tilde{R}_{q}(\tau))\text{d}\tau.
\end{equation}
By Lemmas \ref{lemma:oscillatory_int_1} and \ref{lemma:oscillatory_int_2}, { over  $[0,1/(\epsilon_1\epsilon_2)]$},
\begin{equation}
\label{eq:order_X_5}
X_5(t)=\mathcal{O}\left(\frac{1}{\sqrt{\epsilon_1\epsilon_2}}\right).
\end{equation}
Notice that, by Lemma \ref{lemma:theta_derivative_order} and equation \eqref{eq:def_hat_H_4}, $\hat{H}_{4}(t)=\mathcal{O}(\epsilon_2+\epsilon_1^2)$ on $[0,1/(\epsilon_1\epsilon_2)]$. Then by Proposition \ref{prop:change_of_var}, the dynamics of $\hat{\psi}_{5}$ are characterized by the Hamiltonian
\begin{equation*}
\begin{aligned}
     \hat{H}_{5}(t)=&\hat{H}_{4}(t)+\epsilon_1^2\frac{\text{d}}{\text{d}t}X_{5}(t)-\left[i\epsilon_1^2X_5(t),\hat{H}_{4}(t)+\frac{\epsilon_1^2}{2}\frac{\text{d}}{\text{d}t}X_5(t)\right]+\mathcal{O}(\epsilon_{1}^3)\\
     =&-\frac{\epsilon_1\epsilon_2}{2}\dot{\theta}_{\epsilon_1}(\epsilon_1\epsilon_2 t)B_{pq}(-2\tilde{\phi}(t))+\epsilon_1^2 U_4(t)\Big(\sum_{j=1}^{n}h_{jj}^{0}(\epsilon_1\epsilon_2 t)\textbf{e}_{jj}\Big)U_{4}^{\dagger}(t)\\
     &-\left[i\epsilon_1^2X_5(t),-\frac{\epsilon_1\epsilon_2}{2}\dot{\theta}_{\epsilon_1}(\epsilon_1\epsilon_2 t)B_{pq}(-2\tilde{\phi}(t))\right]+\mathcal{O}(\epsilon_{1}^3+\epsilon_{1}^2\epsilon_2+\epsilon_{1}^{7/2}\epsilon_2^{-1/2}).
\end{aligned}
\end{equation*}
\subsection*{Step 5: Rotating wave approximation}
Let us introduce the truncation of $\hat{H}_{5}(t)$ given by
\begin{equation*}
    \begin{aligned}
     \hat{H}_{\text{RWA}}(t)
     =-\frac{\epsilon_1\epsilon_2}{2}\dot{\theta}_{\epsilon_1}(\epsilon_1\epsilon_2 t)B_{pq}(-2\tilde{\phi}(t))+\epsilon_1^2 U_4(t)\Big(\sum_{j=1}^{n}h_{jj}^{0}(\epsilon_1\epsilon_2 t)\textbf{e}_{jj}\Big)U_{4}^{\dagger}(t).
    \end{aligned}
\end{equation*}
We denote by $\hat{\psi}_{\text{RWA}}(t)$ the solution of 
\begin{equation}
\label{eq:dynamics_RWA}
    i\frac{\text{d}}{\text{d}t}\hat{\psi}_{\text{RWA}}(t)=\hat{H}_{\text{RWA}}(t)\hat{\psi}_{\text{RWA}}(t),\quad\hat{\psi}_{\text{RWA}}(0)=\hat{\psi}_{5}(0).
\end{equation}
\begin{lemma}
    \label{lemma:RWA_1}
    $\left\|\hat{\psi}_{\text{RWA}}\left(t\right)-\hat{\psi}_{5}\left(t\right)\right\|=\mathcal{O}(\epsilon_1^{3/2}\epsilon_2^{-1/2}+\epsilon_1+\epsilon_1^{5/2}\epsilon_{2}^{-3/2})$ over $[0,1/(\epsilon_1\epsilon_2)]$.
\end{lemma}
\begin{proof}
    Introduce the $U(n)$-valued functions $V$ and $W$ solutions of 
    $i\dot V(t)=\hat{H}_{5}(t)V(t)$,
    $i\dot W(t)=\hat{H}_{\text{RWA}}(t)W(t)$ with initial conditions 
    $V(0)=\mathbb{I}_{n}$, $W(0)=\mathbb{I}_{n}$.
It is evident that $\hat{\psi}_{5}(t)=V(t)\hat{\psi}_{5}(0)$ and $\hat{\psi}_{\text{RWA}}(t)=W(t)\hat{\psi}_{5}(0)$ for every $t\in\left[0,1/(\epsilon_1\epsilon_2)\right]$. By differentiating $W^{\dagger}(t)V(t)$ with respect to $t$, we  obtain that
\begin{equation*}
    \frac{\text{d}}{\text{d}t}\left(W^{\dagger}(t)V(t)\right)=-iW^{\dagger}(t)\left(\hat{H}_{5}(t)-\hat{H}_{\text{RWA}}(t)\right)V(t).
\end{equation*}
Then { over  $[0,1/(\epsilon_1\epsilon_2)]$}, we have that
\begin{equation*}
\begin{aligned}
    \left\|W^{\dagger}(t)V(t)\right.\left.-\mathbb{I}_n\right\|\leq&\int_{0}^{t}\left\|\hat{H}_{5}(\tau)-\hat{H}_{\text{RWA}}(\tau)\right\|\text{d}\tau\\
\leq&\epsilon_1^2\int_{0}^{t}\left\|X_5(\tau)\right\|\left|\frac{\epsilon_1\epsilon_2}{2}\dot{\theta}_{\epsilon_1}(\epsilon_1\epsilon_2\tau)\right|\text{d}\tau+\mathcal{O}(\epsilon_1^{2}\epsilon_2^{-1}+\epsilon_1+\epsilon_1^{5/2}\epsilon_2^{-3/2})\\
    \leq&\epsilon_1^2\left\|X_5 \right\|_{\infty}\int_{0}^{\epsilon_1\epsilon_2 t}\frac{1}{2}|\dot{\theta}_{\epsilon_1}(s)|\text{d}s+\mathcal{O}(\epsilon_1^{2}\epsilon_2^{-1}+\epsilon_1+\epsilon_1^{5/2}\epsilon_2^{-3/2}).
\end{aligned}
\end{equation*}
By Lemma \ref{lemma:int_theta} and equation \eqref{eq:order_X_5}, there exist $\eta>0$ and $M>0$ such that, for $\epsilon_1\in(0,\eta)$, we have that, { over  $[0,1/(\epsilon_1\epsilon_2)]$},
\begin{equation*}
\begin{aligned}
 \left\|W^{\dagger}(t)V(t)\right.-\left.\mathbb{I}_n\right\|&<\frac{M\epsilon_1^2}{2}\|X_5\|_{\infty}+\mathcal{O}(\epsilon_1^{2}\epsilon_2^{-1}+\epsilon_1+\epsilon_1^{5/2}\epsilon_2^{-3/2})\\&=\mathcal{O}(\epsilon_1^{3/2}\epsilon_2^{-1/2}+\epsilon_1^{2}\epsilon_2^{-1}+\epsilon_1+\epsilon_1^{5/2}\epsilon_2^{-3/2}).
\end{aligned}
\end{equation*}
Notice that $\epsilon_{1}^{3/2}\epsilon_{2}^{-1/2}+\epsilon_{1}^{5/2}\epsilon_{2}^{-3/2}\geq2\epsilon_{1}^{2}\epsilon_2^{-1}$ for all $\epsilon_1,\epsilon_2>0$.
Then we can deduce that, { over  $[0,1/(\epsilon_1\epsilon_2)]$},
\begin{equation*}
  \left\|W^{\dagger}(t)V(t)-\mathbb{I}_n\right\|=\mathcal{O}(\epsilon_{1}^{3/2}\epsilon_2^{-1/2}+\epsilon_1+\epsilon_1^{5/2}\epsilon_2^{-3/2})
\end{equation*}
and 
\begin{equation*}
\begin{aligned}
    \left\|\hat{\psi}_{\text{RWA}}(t)-\hat{\psi}_{5}(t)\right\|&=\left\|W(t)\hat{\psi}_{5}(0)-V(t)\hat{\psi}_{5}(0)\right\|=\left\|W(t)(\mathbb{I}_n-W^{\dagger}(t)V(t))\hat{\psi}_{5}(0)\right\|\\
    &=\left\|(\mathbb{I}_n-W^{\dagger}(t)V(t))\hat{\psi}_{5}(0)\right\|=\mathcal{O}(\epsilon_{1}^{3/2}\epsilon_2^{-1/2}+\epsilon_1+\epsilon_1^{5/2}\epsilon_2^{-3/2}).
\end{aligned}
\end{equation*}
\end{proof}
For $t\in[0,1/(\epsilon_1\epsilon_2)$], define the unitary transformation
\begin{equation}
\label{eq:def_U_back}
    \hat{U}_{\text{back}}(t)=U_{3}^{\dagger}(t)U_4^{\dagger}(t),
\end{equation}
where $U_3$ and $U_4$ are given as in changes of variables \eqref{eq:third_cv} and \eqref{eq:fourth_cv}. Then, let us introduce the following change of variables that pushes the truncated system backward to the original frame:
\begin{equation*}
    \hat{\psi}_{\text{rwa}}(t)=\hat{U}_{\text{back}}(t)\hat{\psi}_{\text{RWA}}(t).
\end{equation*}
The initial state of $\hat{\psi}_{\text{rwa}}(0)$ is $U_3^{\dagger}(0)U_4^{\dagger}(0)\hat{\psi}_{5}(0)$ and its dynamics are characterized by the Hamiltonian
\begin{equation}
\begin{aligned}
\label{eq:H_rwa}
    \hat{H}_{\text{rwa}}(t)=\epsilon_1\delta_{pq}u(\epsilon_1\epsilon_2t)A_{pq}(\phi_{pq}^{1}(t))+\epsilon_1^2\sum_{j=1}^{n}h_{jj}^{0}(\epsilon_1\epsilon_2t)\textbf{e}_{jj}.
\end{aligned}
\end{equation}
Recall that $\phi_{pq}^{1}(t)=(\lambda_p-\lambda_q)t+\phi(t)$ for $t\in[0,1/(\epsilon_1\epsilon_2)]$.
\begin{lemma}
\label{lemma:RWA_2}
    For $\epsilon_1,\epsilon_2>0$, we have
    \begin{equation*}
    \begin{aligned}
        \hat{\psi}_{\text{rwa}}(0)=&\psi_{I}(0),\\ \hat{\psi}_{\text{rwa}}\left(\frac{1}{\epsilon_1\epsilon_2}\right)=&\psi_{I}\left(\frac{1}{\epsilon_1\epsilon_2}\right)+\mathcal{O}(\epsilon_1^{3/2}\epsilon_2^{-1/2}+\epsilon_1+\epsilon_1^{5/2}\epsilon_{2}^{-3/2}).
    \end{aligned}
    \end{equation*}
\end{lemma}
\begin{proof}
    By Remark \ref{remark:final_state}, we have 
    \begin{equation}
        \label{eq:RWA_2}
        \psi_{I}(0)=\hat{\psi}_{2}(0),\quad\psi_{I}(1/(\epsilon_1\epsilon_2))=\hat{\psi}_{2}(1/(\epsilon_1\epsilon_2)).
    \end{equation}
    On the other hand, for all $t\in[0,1/(\epsilon_1\epsilon_2)]$, we have
    \begin{equation*}
    \begin{aligned}
        \|\hat{\psi}_{2}(t)-\hat{\psi}_{\text{rwa}}(t)\|&=\|\hat{\psi}_{2}(t)-\hat{U}_{\text{back}}(t)\hat{\psi}_{\text{RWA}}(t)\|\\
        &\leq\|\hat{\psi}_{2}(t)-\hat{U}_{\text{back}}(t)\hat{\psi}_4(t)\|+\|\hat{U}_{\text{back}}(t)(\hat{\psi}_{4}(t)-\hat{\psi}_5(t))\|\\&+\|\hat{U}_{\text{back}}(t)(\hat{\psi}_5(t)-\hat{\psi}_{\text{RWA}}(t))\|.\\
    \end{aligned}
    \end{equation*}
    Since, by definition \eqref{eq:def_U_back}, $\hat{U}_{\text{back}}(t)$ is unitary and $\hat{\psi}_{2}(t)=\hat{U}_{\text{back}}(t)\hat{\psi}_{4}(t)$ for all $t\in[0,1/(\epsilon_1\epsilon_2)]$, we deduce that
    \begin{equation}
    \label{eq:RWA_3}
    \begin{aligned}
         \|\hat{\psi}_{2}(t)-\hat{\psi}_{\text{rwa}}(t)\|\leq\|\hat{\psi}_{4}(t)-\hat{\psi}_5(t)\|+\|\hat{\psi}_5(t)-\hat{\psi}_{\text{RWA}}(t)\|.
    \end{aligned}
    \end{equation}
By equations \eqref{eq:fifth_cv}, \eqref{eq:def_X_5}, and \eqref{eq:dynamics_RWA}, we have that $\hat{\psi}_{4}(0)=\hat{\psi}_5(0)$ and $\hat{\psi}_{\text{RWA}}(0)=\hat{\psi}_{5}(0)$. Then 
$\psi_{I}(0)=\hat{\psi}_{2}(0)=\hat{\psi}_{5}(0)$. By equations \eqref{eq:fifth_cv}, \eqref{eq:order_X_5}, and Lemma~\ref{lemma:RWA_1}, we have $\|\hat{\psi}_{4}(1/(\epsilon_1\epsilon_2))-\hat{\psi}_{5}(1/(\epsilon_1\epsilon_2))\|=\mathcal{O}(\epsilon_1^{3/2}\epsilon_2^{-1/2})$ and $\|\hat{\psi}_{5}(1/(\epsilon_1\epsilon_2))-\hat{\psi}_{\text{RWA}}(1/(\epsilon_1\epsilon_2))\|=\mathcal{O}(\epsilon_1^{3/2}\epsilon_2^{-1/2}+\epsilon_1+\epsilon_1^{5/2}\epsilon_{2}^{-3/2})$. 
The conclusion follows from
equations~\eqref{eq:RWA_2} and \eqref{eq:RWA_3}.
\end{proof}
\subsection*{Step 6: Non-standard adiabatic approximation}
For a system characterized by the Hamiltonian $\hat{H}_{\text{rwa}}$ given in equation \eqref{eq:H_rwa}, it is evident that the dynamics in the two-dimensional space $\textbf{span}\big(\textbf{e}_{p},\textbf{e}_{q}\big)$ { are} decoupled from the rest of the system. Let us define the decoupled Hamiltonian. Let us introduce the decoupled Hamiltonian
\begin{equation*}
\begin{aligned}
   \hat{H}^{d}_{\text{rwa}}(t)=\epsilon_1\delta_{pq}u(\epsilon_1\epsilon_2t)A_{pq}(\phi(t))+\epsilon_1^2(h_{pp}^{0}(\epsilon_1\epsilon_2t)\textbf{e}_{pp}+h_{qq}^{0}(\epsilon_1\epsilon_2t)\textbf{e}_{qq})
\end{aligned}
\end{equation*}
By assumptions of Theorem \ref{theorem:main} and the change of variables in equation \eqref{eq:interaction_frame}, 
$\psi_{I}(0)=\textbf{e}_{p}$ is in the two-dimensional space $\textbf{span}\big(\textbf{e}_{p},\textbf{e}_{q}\big)$. By Lemma~\ref{lemma:RWA_2}, $\hat{\psi}_{\text{rwa}}(0)=\psi_{I}(0)$ and $\hat{\psi}_{\text{rwa}}(0)\in\textbf{span}\big(\textbf{e}_{p},\textbf{e}_{q}\big)$. Then the solution $\psi_{\text{rwa}}^{d}$ of
\begin{equation*}
    i\frac{\text{d}}{\text{d}t}\hat{\psi}_{\text{rwa}}^{d}(t)=\hat{H}_{\text{rwa}}^{d}(t)\hat{\psi}_{\text{rwa}}^{d}(t),\quad\hat{\psi}_{\text{rwa}}^{d}(0)=\psi_{I}(0),
\end{equation*}
satisfies 
\begin{equation}
    \label{eq:decoupled_dynamics}
    \hat{\psi}_{\text{rwa}}^{d}(t)=\hat{\psi}_{\text{rwa}}(t),\quad\forall t\in\left[0,\frac{1}{\epsilon_1\epsilon_2}\right].
\end{equation}
 We can then apply Lemma 30 in \cite{ROBIN2022414} to the decoupled 2-level system and obtain that, for the control $\omega_{\epsilon_1,\epsilon_2}$ given in Theorem \ref{theorem:main}, {there exist $C'>0$ and $\eta'>0$ such that, for  $(\epsilon_1,\epsilon_2)\in(0,\eta')^2$ and $\hat{\psi}^{d}_{\text{rwa}}(0)=\psi_{I}(0)= \textbf{e}_{p}$,
\begin{equation*}
    \min_{\theta\in[0,2\pi]}\left\|\hat{\psi}_{\text{rwa}}^{d}\left(\frac{1}{\epsilon_1\epsilon_2}\right)-\exp(i\theta)\textbf{e}_q\right\|\leq C'\frac{\epsilon_{2}}{\epsilon_1}.
\end{equation*}}
\subsection*{Step 7: Combination of rotating wave and adiabatic approximations}
By equation~\eqref{eq:decoupled_dynamics} and Lemma~\ref{lemma:RWA_2}, {there exist $C>0$ and $\eta>0$ such that, for $(\epsilon_1,\epsilon_2)\in(0,\eta)^2$ and $\psi_{I}(0)=\textbf{e}_{p}$,}
\begin{equation*}
\begin{aligned}
\min_{\theta\in[0,2\pi]}\left\|\psi_{I}\left(\frac{1}{\epsilon_1\epsilon_2}\right)-\exp(i\theta)\textbf{e}_q\right\|\leq C(\epsilon_2\epsilon_1^{-1}+\epsilon_{1}^{3/2}\epsilon_2^{-1/2}+\epsilon_1+\epsilon_1^{5/2}\epsilon_2^{-3/2}).
\end{aligned}
\end{equation*}
By the change of variables~\eqref{eq:interaction_frame}, it follows that {there exist $C>0$ and $\eta>0$ such that, for $(\epsilon_1,\epsilon_2)\in(0,\eta)^2$ and $\psi_{\epsilon_1,\epsilon_2}(0)=\textbf{e}_{p}$,}
\begin{equation*}
\begin{aligned}    \min_{\theta\in[0,2\pi]}\left\|\psi_{\epsilon_{1},\epsilon_{2}}\left(\frac{1}{\epsilon_{1}\epsilon_{2}}\right)-\exp(i\theta)\textbf{e}_{q}\right\|\leq C(\epsilon_2\epsilon_1^{-1}+\epsilon_{1}^{3/2}\epsilon_2^{-1/2}+\epsilon_1+\epsilon_1^{5/2}\epsilon_2^{-3/2}).
\end{aligned}
\end{equation*}
\begin{proof}[Proof of Proposition \ref{prop:with_second_order}]
Under the same assumption as in Theorem \ref{theorem:main}, we introduce a change of variables as in equation \eqref{eq:first_cv} where the dynamics are characterized by the Hamiltonian $\hat{H}_1$ given in equation \eqref{eq:H_1}. 
Define the set \begin{equation}
    \mathcal{J}''=\mathcal{J}\setminus\left\{(j,j,0)\mid 1\leq j\leq n\right\},
\end{equation}
where $\mathcal{J}$ is defined in equation \eqref{eq:def_J}. The additional assumptions of Proposition \ref{prop:with_second_order} ensure that, for every $(j,k,\sigma)\in\mathcal{J}''$,
\begin{equation*}
    f_{jk}^{\sigma}(t)=\lambda_{j}-\lambda_{k}+\sigma f(\epsilon_1\epsilon_2 t)\neq 0,\quad\forall t\in\left[0,\frac{1}{\epsilon_1\epsilon_2}\right].
\end{equation*}
Then we introduce the following change of variables:
\begin{equation*}
    \hat{\psi}_{1}(t)=\exp(i\epsilon_{1}^{2}\tilde{X}_{2}(t))\tilde{\psi}_{2}(t),
\end{equation*}
where
\begin{equation*}
     \tilde{X}_{2}(t)=\sum_{(j,k,\sigma)\in\mathcal{J}''}\frac{h_{jk}^{\sigma}(\epsilon_{1}\epsilon_{2}t)}{f_{jk}^{\sigma}(\epsilon_{1}\epsilon_{2}t)}B_{jk}\big(\phi_{jk}^{\sigma}(t)\big).
\end{equation*}
Since $h_{jk}^{\sigma}(0)=h_{jk}^{\sigma}(1)=0$ for every $(j,k,\sigma)\in\mathcal{J}''$, we have $\tilde{X}_{2}(0)=\tilde{X}_{2}\left(\frac{1}{\epsilon_1\epsilon_2}\right)=0$. Therefore $\tilde{\psi}_{2}(0)=\hat{\psi}_1(0)=\psi_{I}(0)$ and $\tilde{\psi}_2\left(\frac{1}{\epsilon_1\epsilon_2}\right)=\hat{\psi}_1\left(\frac{1}{\epsilon_1\epsilon_2}\right)=\psi_I\left(\frac{1}{\epsilon_1\epsilon_2}\right)$. By Proposition \ref{prop:change_of_var}, the dynamics of $\hat{\psi}_{2}$ are characterized by the Hamiltonian
\begin{equation*}
\begin{aligned}
    \tilde{H}_{2}(t)=&\hat{H}_{1}(t)+\epsilon_{1}^{2}\frac{\text{d}}{\text{d}t}\tilde{X}_{2}(t)+\mathcal{O}(\epsilon_{1}^{3})\\
=&\epsilon_1\delta_{pq}u(\epsilon_{1}\epsilon_{2}t)A_{pq}\big(\phi_{pq}^{1}(t)\big)+\sum_{j=1}^{n} \epsilon_1^2 h_{jj}^{0}(\epsilon_1\epsilon_2 t)\textbf{e}_{jj}+\mathcal{O}(\epsilon_1^{3}+\epsilon_1^{2}\epsilon_2).
\end{aligned}   
\end{equation*}
Then we introduce the truncation of $\tilde{H}_2(t)$ given by
\begin{equation*}
\begin{aligned}
  \tilde{H}_{\text{rwa}}(t)=\epsilon_1\delta_{pq}u(\epsilon_{1}\epsilon_{2}t)A_{pq}\big(\phi_{pq}^{1}(t)\big)+\sum_{j=1}^{n} \epsilon_1^2 h_{jj}^{0}(\epsilon_1\epsilon_2 t)\textbf{e}_{jj}.
\end{aligned}
\end{equation*}
Denote by $\tilde{\psi}_{\text{rwa}}(t)$ the solution of 
\begin{equation}
\label{eq:dynamics_rwa}
    i\frac{\text{d}}{\text{d}t}\tilde{\psi}_{\text{rwa}}(t)=\tilde{H}_{\text{rwa}}(t)\tilde{\psi}_{\text{rwa}}(t),\quad\tilde{\psi}_{\text{rwa}}(0)=\tilde{\psi}_{2}(0).
\end{equation}
Using a similar reasoning as in the proof of Lemma 
\ref{lemma:RWA_1}, we prove that $ \left\|\hat{\psi}_{\text{rwa}}\left(\frac{1}{\epsilon_1\epsilon_2}\right)-\tilde{\psi}_{I}\left(\frac{1}{\epsilon_1\epsilon_2}\right)\right\|=\mathcal{O}\left(\frac{\epsilon_1^2}{\epsilon_2}+\epsilon_1\right)$.
Then by adiabatic following argument similar to that in the proof of Theorem \ref{theorem:main}, we conclude that there exist $C>0$ and $\eta>0$ such that, for $(\epsilon_1,\epsilon_2)\in(0,\eta)^2$ and $\psi_{\epsilon_1,\epsilon_2}(0)=\textbf{e}_{p}$, 
\begin{equation*}
\begin{aligned}
\min_{\theta\in[0,2\pi]}&\left\|\psi_{I}\left(\frac{1}{\epsilon_1\epsilon_2}\right)-\exp(i\theta)\textbf{e}_q\right\|\leq C\left(\frac{\epsilon_1^2}{\epsilon_2}+\frac{\epsilon_2}{\epsilon_1}\right),\\
\min_{\theta\in[0,2\pi]}&\left\|\psi_{\epsilon_1,\epsilon_2}\left(\frac{1}{\epsilon_1\epsilon_2}\right)-\exp(i\theta)\textbf{e}_q\right\|\leq C\left(\frac{\epsilon_1^2}{\epsilon_2}+\frac{\epsilon_2}{\epsilon_1}\right).
\end{aligned}
\end{equation*}
\end{proof}
\section{Example}
\label{sec:numerical}
Consider the four-level system with drift and control Hamiltonians 
\begin{equation*}
        H(\alpha)=
        \begin{pmatrix}
                         0 & 0 & 0 & 0\\
            0 & 1+\alpha & 0 & 0\\
            0 & 0 & 3+2\alpha & 0\\
            0 & 0 & 0 & 7
        \end{pmatrix}, 
    \ H_c=
    {\begin{pmatrix}
        1 & 1 & 1 & 0\\
        1 & 1 & 2 & 0\\
        1 & 2 & 1 & 3\\
        0 & 0 & 3 & 1
    \end{pmatrix}.}
\end{equation*}
Fix the initial state $\psi_{\epsilon_1,\epsilon_2}(0)=(0,0,1,0)^{\top}$. In order to realize a population inversion between the third and the forth eigenstates, let us fix $T=1$ and use the control law given in Remark~\ref{remark:control_law_construction} { with $(v_0,v_1)=(3,5)$.} We will test the sharpness of conditions for $\alpha\in\{-0.6,-0.3,-0.1,0.1,0.3\}$. We are going to use the time scale $(\epsilon_1,\epsilon_2)=(10^{-5/3},10^{-7/3})$. For $s\in[0,1]$, let us define the fidelity of the inversion as the population on the target state $\textbf{e}_{q}$ at instant $t=s/(\epsilon_1\epsilon_2)$:
\begin{equation*}
    \text{fid}(s)=\left|\left\langle\psi_{\epsilon_1,\epsilon_2}\left(\frac{1}{\epsilon_1\epsilon_2}\right),\textbf{e}_{q}\right\rangle\right|^{2}.
\end{equation*}
Here the target state is $\textbf{e}_{4}=(0,0,0,1)^{\top}$. When $\alpha=-0.3$ or $\alpha=-0.1$, $\lambda_{4}(\alpha)-\lambda_{3}(\alpha)$ falls in the open interval $(v_0,v_1)$, and for $1\leq j<k\leq 4$ with $(j,k)\neq(3,4)$, $\lambda_{k}(\alpha)-\lambda_{j}(\alpha)$ is not in the closed interval $[v_0,v_1]$. When $\alpha=-0.6$, $\lambda_{4}(\alpha)-\lambda_{3}(\alpha)$ is not in $(v_0,v_1)$, violating the hypothesis \ref{cond_1} of Theorem \ref{theorem:main}. When $\alpha=0.1$ or $\alpha=0.3$, $\lambda_{3}(\alpha)-\lambda_{1}(\alpha)$ falls in the closed interval $[v_0,v_1]$ and the hypothesis \ref{cond_2} of Theorem~\ref{theorem:main} is violated. See the results in Figure~\ref{fig:4_level_example_1}.\\
\begin{figure}[thp]
    \centering
    \begin{subfigure}{0.4\textwidth}
    \includegraphics[width=\textwidth]{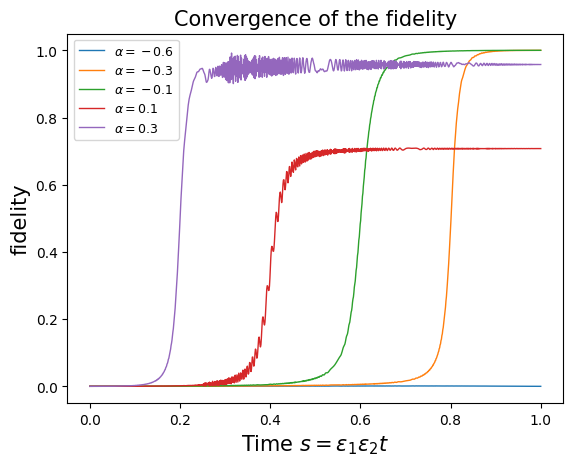}
    \caption{Convergence of $\text{fid}(s)$ }
\label{fig:simulation_1}
    \end{subfigure}
    \hfill
    \begin{subfigure}{0.4\textwidth}
    \includegraphics[width=\textwidth]{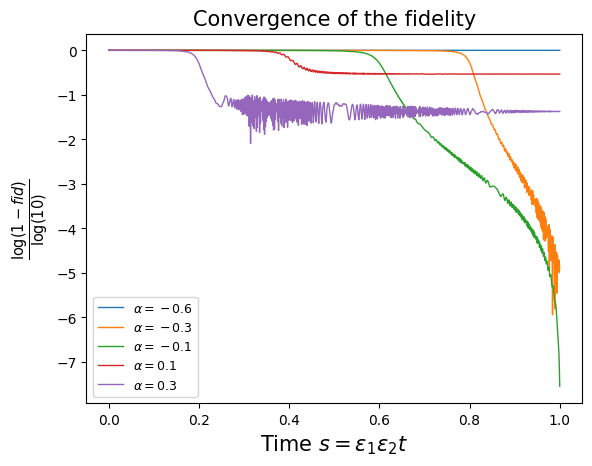}
    \caption{Logarithmic convergence of $1-\text{fid}(s)$}
\label{fig:simulation_2}
\end{subfigure}
\caption{Simulations for $\epsilon_1=10^{-5/3}$, $\epsilon_2=10^{-7/3}$, and $\alpha\in\{-0.6,-0.3,-0.1,0.1,0.3\}$: To better illustrate the convergence, we also show the logarithmic convergence of $1-\text{fid}(s)$, specifically $\log(1-\text{fid}(s))/\log10$. The convergence is only guaranteed when the hypotheses of Theorem~\ref{theorem:main} are satisfied. When 
$\lambda_{4}(\alpha)-\lambda_{3}(\alpha)$ is not in $(v_0,v_1)$, the population inversion between the third and the fourth eigenstates does not occur at all (see the blue curve). When $\lambda_{3}(\alpha)-\lambda_{1}(\alpha)$ is in $[v_0,v_1]$, {  the fidelity stays quite far from 1 and a total population inversion is not accomplished} (see the purple and red curves).} 
\label{fig:4_level_example_1}
\end{figure}
Finally, let us test if it is possible to realize a population inversion between $\textbf{e}_1$ and $\textbf{e}_4$ by successive inversions between $(\textbf{e}_1,\textbf{e}_2)$, between $(\textbf{e}_2,\textbf{e}_3)$, and between $(\textbf{e}_3,\textbf{e}_4)$. Notice that a direct inversion between $\textbf{e}_1$ and $\textbf{e}_4$ is impossible since $\delta_{14}=\langle \textbf{e}_1,H_{c}\textbf{e}_{4}\rangle=0$. Fix $\alpha=-0.1$ for this test. To realize the inversion between $\textbf{e}_1$ and $\textbf{e}_2$, let us choose $(v_0,v_1)=(0.5,1.5)$ and construct a control law $\omega_{\epsilon_1,\epsilon_2}^{1}$
defined on $[0,1/(\epsilon_1\epsilon_2)]$ as described in Remark \ref{remark:control_law_construction}. It can be easily verified that the assumptions in Theorem \ref{theorem:main} are satisfied. Similarly, we choose $(v_0,v_1)=(1.5,2.5)$ to construct $\omega_{\epsilon_1,\epsilon_2}^{2}$ for the inversion between $(\textbf{e}_2,\textbf{e}_3)$ and $(v_0,v_1)=(3,5)$ to construct $\omega_{\epsilon_1,\epsilon_2}^{3}$ for the inversion between $(\textbf{e}_3,\textbf{e}_4)$. A concatenation of these three control laws, defined on $\left[0,\frac{3}{\epsilon_1\epsilon_2}\right]$, is given by
\begin{equation*}
    \omega_{\epsilon_1,\epsilon_2}: t\mapsto\left\{\begin{matrix}
       \omega_{\epsilon_{1}\epsilon_{2}}^{1}(t),& \text{if }t\in\left[0,\frac{1}{\epsilon_1\epsilon_2}\right),\\
       \omega_{\epsilon_{1}\epsilon_{2}}^{2}\left(t-\frac{1}{\epsilon_1\epsilon_2}\right),& \text{if }t\in\left[\frac{1}{\epsilon_1\epsilon_2},\frac{2}{\epsilon_1\epsilon_2}\right),\\
       \omega_{\epsilon_{1}\epsilon_{2}}^{3}\left(t-\frac{2}{\epsilon_1\epsilon_2}\right),& \text{if }t\in\left[\frac{2}{\epsilon_1\epsilon_2},\frac{3}{\epsilon_1\epsilon_2}\right].
    \end{matrix}\right.
\end{equation*}
\begin{figure}[htp]
\includegraphics[width=0.6\linewidth]{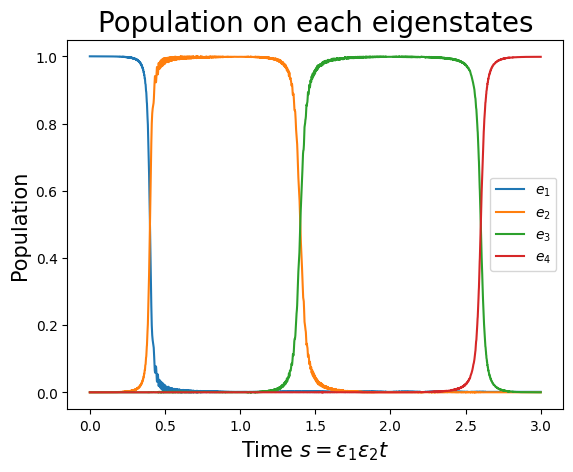}
\centering
\caption{Simulations for $\alpha=-0.1$, $\epsilon_{1}=10^{-5/3},\epsilon_{2}=10^{-7/3}$: Population inversions between $(\textbf{e}_1,\textbf{e}_2)$, $(\textbf{e}_2,\textbf{e}_3)$, and $(\textbf{e}_3,\textbf{e}_4)$ happen successively.}
\label{fig:population_inversion}
\end{figure}
The numerical result is given in Figure \ref{fig:population_inversion} with the time scale $(\epsilon_1,\epsilon_2)=(10^{-5/3},10^{-7/3})$, showing the efficiency of the algorithm.
\section{Conclusion}
In this study, we introduced an algorithm capable of realizing population inversion between two arbitrary eigenstates for a continuum of quantum systems. We underlined the importance of the non-overlapping condition on some characteristic frequencies for this algorithm's validity.
Future investigations could explore the possibility of proposing weaker conditions for convergence, extending similar results to infinite-dimensional quantum systems and examining which further controllability results (i.e. population splitting) could be obtained in addition to population inversion.
\begin{section}{Appendix}
\begin{proof}[Proof of Lemma \ref{lemma:int_theta}]
   By the definition of $\theta_{\epsilon_1}$ in \ref{eq:def_theta}, we have that
   \begin{equation}
   \label{eq:angular_acceleration}
       \dot{\theta}_{\epsilon}(s)=\delta_{pq}\epsilon_1\frac{(\Delta-f(s))\dot{u}(s)+\dot{f}(s)u(s)}{2\lambda_{\epsilon_1}(s)^2}.
   \end{equation}
    The conditions on $u$ and $f$ in Theorem \ref{theorem:main} ensure that there exist $\sigma>0$ and $M_1>0$ such that  $\left[\bar{s}(\alpha)-\sigma,\bar{s}(\alpha)+\sigma\right]\subset[0,1]$,
        $(\Delta-f(s))\dot{u}(s)+\dot{f}(s)u(s)>0$ for all $s\in\left[\bar{s}(\alpha)-\sigma,\bar{s}(\alpha)+\sigma\right]$,
    and $|\Delta-f(s)|>M_1$ for all $s\in[0,1]\setminus\left[\bar{s}(\alpha)-\sigma,\bar{s}(\alpha)+\sigma\right]$.
    Then, by equation \eqref{eq:angular_acceleration}, we deduce that $\text{sgn}(\dot{\theta}_{\epsilon_1}(s))=\text{sgn}(\delta_{pq})$ for all $s\in[\bar{s}(\alpha)-\sigma,\bar{s}(\alpha)+\sigma]$. Since $\dot{\theta}_{\epsilon}(\cdot)$ does not change sign on $[\bar{s}(\alpha)-\sigma,\bar{s}(\alpha)+\sigma]$, we obtain that
    \begin{equation*}
    \begin{aligned}
        \int_{\bar{s}(\alpha)-\sigma}^{\bar{s}(\alpha)+\sigma}\left|\dot{\theta}_{\epsilon_1}(s)\right|\text{d}s=&\left|\int_{\bar{s}(\alpha)-\sigma}^{\bar{s}(\alpha)+\sigma}\dot{\theta}_{\epsilon_1}(s)\text{d}s\right|=\left|\theta_{\epsilon_1}(\bar{s}(\alpha)+\sigma)-\theta_{\epsilon_1}(\bar{s}(\alpha)-\sigma)\right|\leq\pi.
    \end{aligned}
    \end{equation*}
   For all $s\in[0,1]\setminus\left[\bar{s}(\alpha)-\sigma,\bar{s}(\alpha)+\sigma\right]$, by  equation~\eqref{eq:def_lambda}, we have $ \lambda_{\epsilon_1}(s)>\frac{1}{2}\left|\Delta-f(s)\right|\geq\frac{1}{2}M_1$. Then by equation \eqref{eq:angular_acceleration}, we deduce that there exists $M_2>0$ such that $|\dot{\theta}_{\epsilon_1}(s)|\leq M_2\epsilon_1$ for all $s\in[0,1]\setminus\left[\bar{s}(\alpha)-\sigma,\bar{s}(\alpha)+\sigma\right]$. Therefore,
\begin{equation*}
\begin{aligned}
    \int_{0}^{1}\left|\dot{\theta}_{\epsilon_1}(s)\right|\text{d}s=&{}\int_{0}^{\bar{s}(\alpha)}\left|\dot{\theta}_{\epsilon_1}(s)\right|\text{d}s+\int_{\bar{s}(\alpha)-\sigma}^{\bar{s}(\alpha)+\sigma}\left|\dot{\theta}_{\epsilon_1}(s)\right|\text{d}s+\int_{\bar{s}(\alpha)+\sigma}^{1}\left|\dot{\theta}_{\epsilon_1}(s)\right|\text{d}s\\ <{}&\pi+(1-2\delta)M_2\epsilon_1.
\end{aligned}
\end{equation*}
We conclude that there exist $\eta>0$ and $M>0$ such that, for all $\epsilon_1\in(0,\eta)$, we have $ \int_{0}^{1}\left|\dot{\theta}_{\epsilon_1}(s)\right|\text{d}s<M$.
\end{proof}
\begin{proof}[Proof of Lemma \ref{lemma:function_integral_1}]
For $\epsilon_1>0$ and $s\in[0,1]$, we have
    \begin{equation*}
    \begin{aligned}
    \int_{0}^{s}\left|\dot{F}_{\epsilon_1}(\xi)\right|\text{d}\xi\leq&\int_{0}^{s}\left|g(\theta_{\epsilon_1}(\xi))\dot{h}(\xi)\right|\text{d}\tau+\int_{0}^{s}\left|\dot{g}(\theta_{\epsilon_1}(\xi))\dot{\theta}_{\epsilon_1}(\xi)h(\xi)\right|\text{d}\xi
    \\
\leq&\|\dot{g}\|_{\infty}\|h\|_{\infty}\int_{0}^{s}\left|\dot{\theta}_{\epsilon_1}(\xi)\right|\text{d}\xi+s\|g\|_{\infty}\|\dot{h}\|_{\infty}.
    \end{aligned}
    \end{equation*}
    By Lemma \ref{lemma:int_theta}, there exists $M>0$ such that, for all $s\in[0,1]$, $ \int_{0}^{s}\left|\dot{\theta}_{\epsilon_1}(\xi)\right|\text{d}\xi\leq\int_{0}^{1}\left|\dot{\theta}_{\epsilon_1}(\xi)\right|\text{d}\xi<M$. We can conclude the proof by choosing $c=M\|\dot{g}\|_{\infty}\|h\|_{\infty}+\|g\|_{\infty}\|\dot{h}\|_{\infty}$.
\end{proof}
\begin{proof}[Proof of Lemma \ref{lemma:van_der_corput_1}]
    By equation \eqref{eq:def_lambda},
    \begin{equation*}
         -2\dot{\lambda}_{\epsilon_1}(s)=\frac{(\Delta(\alpha)-f(s))\dot{f}(s)-4\epsilon_1^2\delta_{pq}^2u(s)\dot{u}(s)}{2\lambda_{\epsilon_1}(s)}.
    \end{equation*}
    Take $s\in[0,\bar{s}(\alpha)]$. By the definition of $\bar{s}$ in \eqref{eq:def_s_bar} and the assumption of Theorem \ref{theorem:main} on $f$, we have $\Delta(\alpha)-f(s)\geq 0$ and $\dot{f}(s)>0$. Then, we can easily deduce that
\begin{equation*}\begin{aligned}-2\dot{\lambda}_{\epsilon_1}(s)+\dot{f}(s)\geq{}&\frac{-4\epsilon_1^2\delta_{pq}^2u(s)\dot{u}(s)}{\lambda_{\epsilon_1}(s)}+\dot{f}(s)\\\geq{}&\frac{-4\epsilon_1^2\delta_{pq}^2u(s)\|\dot{u}\|_{\infty}}{\sqrt{(\Delta(\alpha)-f(s))^2+4\epsilon_1^2\delta_{pq}^2 u(s)^2}}+\dot{f}(s)\\\geq{}&-2\epsilon_1|\delta_{pq}|\|\dot{u}\|_{\infty}+\dot{f}(s).
    \end{aligned}
    \end{equation*}
Take $s\in[\bar{s}(\alpha),1]$. By equation \eqref{eq:def_lambda} and triangle inequality, we have that
\begin{equation*}
    2\lambda_{\epsilon_1}(s)\leq\left|\Delta(\alpha)-f(s)\right|+2\epsilon_1|\delta_{pq}|u(s).
\end{equation*}
Since $f(s)\geq\Delta(\alpha)$ on $[\bar{s}(\alpha),1]$, we can deduce that
\begin{equation*}
    \begin{aligned}
        -2\lambda_{\epsilon_1}(s)+f(s)&\geq -\left|\Delta(\alpha)-f(s)\right|-2\epsilon_1|\delta_{pq}|u(s)+f(s)\\
        &=-(f(s)-\Delta(\alpha))+f(s)-2\epsilon_1|\delta_{pq}|u(s)=\Delta(\alpha)-2\epsilon_1|\delta_{pq}|u(s).
    \end{aligned}
\end{equation*}
By the assumptions of Theorem \ref{theorem:main}, there exists $v_0>0$ such that $\Delta(\alpha)>v_0$ for all $\alpha\in\mathcal{D}$. We conclude the proof by choosing $ \eta=\min\left(\frac{\min|\dot{f}(s)|}{4|\delta_{pq}|\|\dot{u}\|_{\infty}},\frac{v_0}{4|\delta_{pq}|\|u\|_{\infty}}\right)$.
\end{proof}
\begin{proof}[Proof of Lemma \ref{lemma:van_der_corput_2}]
    By equation \eqref{eq:def_lambda} and Cauchy-Schwarz inequality, we have that
    \begin{equation*}
    \begin{aligned}
        \left|\dot{\lambda}_{\epsilon_1}(s)\right|={}&\left|\frac{-\frac{1}{2}(\Delta(\alpha)-f(s))\dot{f}(s)+2\epsilon_1^2\delta_{pq}^2u(s)\dot{u}(s)}{\sqrt{(\Delta(\alpha)-f(s))^2+4\epsilon_1^2\delta_{pq}^2 u(s)^2}}\right|\\
        \leq {}& \sqrt{\frac{1}{4}(\dot{f}(s))^2+\epsilon_1^2\delta_{pq}^2(\dot{u}(s))^2}\leq {} \frac{1}{2}\left|\dot{f}(s)\right|+\epsilon_1|\delta_{pq}|\|\dot{u}\|_{\infty}.
    \end{aligned}
    \end{equation*}
    We conclude the proof by choosing $\eta=\frac{\min\dot{f}(s)}{2|\delta_{pq}|\|\dot{u}\|_{\infty}}$.
\end{proof}
\end{section}
\section*{Acknowledgments}                          
\noindent This work has been partly supported by  the ANR-DFG project “CoRoMo”
ANR-22-CE92-0077-01. This project has received financial support from the
CNRS through the MITI interdisciplinary programs. 
\bibliographystyle{plain}  
\bibliography{reference}

\begin{thebibliography}{10}

\bibitem{AugierSIAM}
Nicolas Augier, Ugo Boscain, and Mario Sigalotti.
\newblock Adiabatic ensemble control of a continuum of quantum systems.
\newblock {\em SIAM J. Control Optim.}, 56(6):4045--4068, 2018.

\bibitem{AugierMCRF}
Nicolas Augier, Ugo Boscain, and Mario Sigalotti.
\newblock Semi-conical eigenvalue intersections and the ensemble controllability problem for quantum systems.
\newblock {\em Math. Control Relat. Fields}, 10(4):877--911, 2020.

\bibitem{augier2022effective}
Nicolas Augier, Ugo Boscain, and Mario Sigalotti.
\newblock Effective adiabatic control of a decoupled hamiltonian obtained by rotating wave approximation.
\newblock {\em Automatica}, 136:110034, 2022.

\bibitem{BeauchardCoronRouchon}
Karine Beauchard, Jean-Michel Coron, and Pierre Rouchon.
\newblock Controllability issues for continuous-spectrum systems and ensemble controllability of {B}loch equations.
\newblock {\em Comm. Math. Phys.}, 296(2):525--557, 2010.

\bibitem{long000}
Chunlin Chen, Daoyi Dong, Ruixing Long, Ian~R Petersen, and Herschel~A Rabitz.
\newblock Sampling-based learning control of inhomogeneous quantum ensembles.
\newblock {\em Physical Review A}, 89(2):023402, 2014.

\bibitem{ChittaroGauthier}
Francesca~C. Chittaro and Jean-Paul Gauthier.
\newblock Asymptotic ensemble stabilizability of the {B}loch equation.
\newblock {\em Systems Control Lett.}, 113:36--44, 2018.

\bibitem{rev1}
S.~J. {Glaser}, U.~{Boscain}, T.~{Calarco}, C.~P. {Koch}, W.~{K{\"o}ckenberger}, R.~{Kosloff}, I.~{Kuprov}, B.~{Luy}, S.~{Schirmer}, T.~{Schulte-Herbr{\"u}ggen}, D.~{Sugny}, and F.~K. {Wilhelm}.
\newblock {Training {S}chr{\"o}dinger's cat: quantum optimal control. Strategic report on current status, visions and goals for research in {E}urope}.
\newblock {\em European Physical Journal D}, 69:279, 2015.

\bibitem{glaser2015training}
Steffen~J Glaser, Ugo Boscain, Tommaso Calarco, Christiane~P Koch, Walter K{\"o}ckenberger, Ronnie Kosloff, Ilya Kuprov, Burkhard Luy, Sophie Schirmer, Thomas Schulte-Herbr{\"u}ggen, et~al.
\newblock Training schr{\"o}dinger’s cat: Quantum optimal control: Strategic report on current status, visions and goals for research in europe.
\newblock {\em The European Physical Journal D}, 69:1--24, 2015.

\bibitem{hall2000elementary}
Brian~C. Hall.
\newblock An elementary introduction to groups and representations, 2000.

\bibitem{rev2}
Christiane~P Koch, Ugo Boscain, Tommaso Calarco, Gunther Dirr, Stefan Filipp, Steffen~J Glaser, Ronnie Kosloff, Simone Montangero, Thomas Schulte-Herbr{\"u}ggen, Dominique Sugny, et~al.
\newblock Quantum optimal control in quantum technologies. strategic report on current status, visions and goals for research in europe.
\newblock {\em EPJ Quantum Technology}, 9(1):19, 2022.

\bibitem{LeghtasSarletteRouchon}
Z.~Leghtas, A.~Sarlette, and P.~Rouchon.
\newblock Adiabatic passage and ensemble control of quantum systems.
\newblock {\em Journal of Physics B}, 44(15), 2011.

\bibitem{li2006ensemble}
Jr-Shin Li and Navin Khaneja.
\newblock Ensemble controllability of the {B}loch equations.
\newblock In {\em Proceedings of the 45th IEEE Conference on Decision and Control}, pages 2483--2487. IEEE, 2006.

\bibitem{li2009ensemble}
Jr-Shin Li and Navin Khaneja.
\newblock Ensemble control of {B}loch equations.
\newblock {\em IEEE Transactions on Automatic Control}, 54(3):528--536, 2009.

\bibitem{LiangBoscainSigalotti-CDC2024}
Ruikang {Liang}, Ugo {Boscain}, and Mario {Sigalotti}.
\newblock Ensemble quantum control with a scalar input.
\newblock In {\em Proceedings of the 63rd IEEE Conference on Decision and Control}. IEEE, 2024.

\bibitem{Netoetal}
Ulisses~Alves Maciel~Neto, Paulo~Sergio Pereira~da Silva, and Pierre Rouchon.
\newblock Motion planing for an ensemble of {B}loch equations towards the south pole with smooth bounded control.
\newblock {\em Automatica J. IFAC}, 145:Paper No. 110529, 9, 2022.

\bibitem{ROBIN2022414}
Rémi Robin, Nicolas Augier, Ugo Boscain, and Mario Sigalotti.
\newblock Ensemble qubit controllability with a single control via adiabatic and rotating wave approximations.
\newblock {\em Journal of Differential Equations}, 318:414--442, 2022.

\bibitem{stein1993harmonic}
Elias~M Stein and Timothy~S Murphy.
\newblock {\em Harmonic analysis: real-variable methods, orthogonality, and oscillatory integrals}, volume~3.
\newblock Princeton University Press, 1993.

\bibitem{Teufel}
Stefan Teufel.
\newblock {\em Adiabatic perturbation theory in quantum dynamics}, volume 1821 of {\em Lecture Notes in Mathematics}.
\newblock Springer-Verlag, Berlin, 2003.

\end{thebibliography}
\end{document}